\documentclass[12pt]{amsart}
\usepackage{amsmath,amsfonts,amsthm}
\usepackage{latexsym}
\usepackage{graphicx}
\usepackage{subfigure}
\usepackage{color}

\newtheorem{theorem}{Theorem}[section]
\newtheorem{proposition}[theorem]{Proposition}
\newtheorem{lemma}[theorem]{Lemma}

\newtheorem{corollary}[theorem]{Corollary}
\theoremstyle{definition}

\newtheorem{example}[theorem]{Example}

\theoremstyle{remark}
\newtheorem{remark}[theorem]{Remark}

\numberwithin{equation}{section}
\allowdisplaybreaks

\newcommand{\C}{\mathbb C}
\newcommand{\R}{\mathbb R}

\newcommand{\N}{\mathbb N}

\newcommand{\Sb}{\mathbb{S}}
\newcommand{\Eb}{\mathbb E}

\newcommand{\Fc}{\mathcal{F}}

\newcommand{\fr}{\mathcal{F}}

\newcommand{\ip}[2]{\langle#1,#2\rangle}
\newcommand{\<}{\langle}
\renewcommand{\>}{\rangle}
\newcommand{\tr}{\text{Tr}}
\newcommand{\diag}{\operatorname{diag}}

\newcommand{\Vol}{\operatorname{Vol}}

\newcommand{\sca}{\mathcal{SC}}
\newcommand{\co}{\operatorname{co}}

\newcommand{\vphi}{\varphi}
\newcommand{\veps}{\varepsilon}
\newcommand{\FP}{\mathbb F\mathbb P}
\newcommand{\one}{{\mathbf 1}}

\newcommand{\Sym}{\operatorname{Sym}}
\newcommand{\Diag}{\operatorname{Diag}}

\title{Measures of scalability}
\author{Xuemei Chen}
\address{Xuemei Chen\\
Department of Mathematics\\
University of Maryland\\
College Park, MD 20742 USA}
\email{xuemeic@math.umd.edu}

\author{Gitta Kutyniok}
\address{Gitta Kutyniok\\
Institut f\"ur Mathematik\\
Technische Universit\"at Berlin \\
Strasse des 17. Juni 136\\
10623 Berlin, Germany}
\email{kutyniok@math.tu-berlin.de}

\author{Kasso A.~Okoudjou}
\address{Kasso A.~Okoudjou\\
Department of Mathematics\\
University of Maryland\\
College Park, MD 20742 USA}
\email{kasso@math.umd.edu}

\author{Friedrich Philipp}
\address{Friedrich Philipp\\
Institut f\"{u}r Mathematik\\
Technische Universit\"{a}t Berlin\\
Strasse des 17. Juni 136\\
D 10623 Berlin, Germany}
\email{philipp@math.tu-berlin.de}

\author{Rongrong Wang}
\address{Rongrong Wang\\
Department of Mathematics\\
University of British Columbia\\
Vancouver, BC V6T1Z2 Canada}
\email{rongwang@math.ubc.ca}


\begin{document}
\begin{abstract}
Scalable frames are frames with the property that the frame vectors can be rescaled resulting in tight frames. However, if a frame is not scalable, one has to aim for an approximate procedure. For this, in this paper we introduce three novel quantitative measures of the closeness to scalability for frames in finite dimensional real Euclidean spaces. Besides the natural measure of scalability given by the distance of a frame to the set of scalable frames, another measure is obtained by optimizing a quadratic functional, while the third is given by the volume of the ellipsoid of minimal volume containing the symmetrized frame. After proving that these measures are equivalent in a certain sense, we establish bounds on the probability of a randomly selected frame to be scalable. In the process, we also derive new necessary and sufficient conditions for a frame to be scalable.
\end{abstract}

\date{\today}

\maketitle  \pagestyle{myheadings} \thispagestyle{plain}
\markboth{X. CHEN, G. KUTYNIOK, K.  A. OKOUDJOU, F.  PHILIPP, R. WANG}{MEASURES OF SCALABILITY}

\section{Introduction}\label{sec1}
During the last years, frames have had a tremendous impact on applications due to their unique ability to deliver redundant, yet stable expansions. The redundancy of a frame is typically utilized by applications which either require robustness of the frame coefficients to noise, erasures, quantization, etc. or  require sparse expansions in the frame.
More precisely, letting $\Phi=\{\varphi_i\}_{i=1}^{M}\subset \R^N$ be a frame, either {\it decompositions} into a sequence of frame coefficients of a signal $x \in \R^N$, which is the image of $x$ under the analysis operator $T : \R^N \to \R^M$, $x \mapsto (\langle x , \varphi_i \rangle)_{i=1}^M$, are exploited by applications such as telecommunications and imaging sciences, or {\it expansions} in terms of the frame, i.e., $x = \sum_{i=1}^M c_i \varphi_i$ with suitable choice of coefficients $(c_i)_{i=1}^M$, are required by applications such as efficient PDE solvers. Intriguingly, the novel area of compressed sensing is based on the fact that typically signals exhibit a sparse expansion in a frame, which is nowadays considered the standard paradigm in data processing. Some compressed sensing applications also `hope' that the sequence of frame coefficients itself is sparse; a connection deeply studied in a series of papers on {\it cosparsity} (cf. \cite{NDEG13}).

The discussed applications certainly require stability, numerically as well as  theoretically. For instance, notice 
that most results in compressed sensing are stated for tight frames, i.e., for optimal stability. It is known that such frames -- in the case
of normalized vectors -- can be characterized by the frame potential (see, e.g., \cite{BF,cfm11,fjko})
and construction methods have been derived (cf. \cite{cfmmn} and \cite{str12} for an algebro-geometric point of view).
However, a crucial question remains: Given a frame with desirable properties, can we turn it into a tight frame? The immediate answer is yes, since it can easily be shown that applying $S^{-1/2}$ to each frame element, $S : \R^N \to \R^N$ denoting the frame operator $Sx = \sum_{i=1}^M \langle x , \varphi_i \rangle \varphi_i$, produces a Parseval frame. Thinking further one however realizes a serious problem with this seemingly elegant approach; it typically completely destroys any properties of the frame for which it was carefully designed before. Thus, unless we are merely interested in theoretical considerations, this approach is unacceptable.

Trying to be as careful as possible, the most noninvasive approach seems  to merely scale
each frame vector, i.e., multiply it by a scalar. And, indeed, almost all frame properties one can
think of such as erasure resilience or sparse expansions are left untouched by this modification.
In fact, this approach is currently extensively studied under the theme of scalable frames.

\subsection{Scalability of Frames}
The notion of a {\it scalable frame} was first introduced in \cite{kopt13} as a frame whose frame vectors
can be rescaled to yield a tight frame. Recall that a sequence $\Phi=\{\varphi_i\}_{i=1}^{M}\subset \R^N$ forms a {\it frame} provided that
$$
A\|x\|^2 \leq \sum_{i=1}^{M}|\ip{x}{\varphi_i}|^{2}\leq B\|x\|^2
$$
for all $x\in\R^N$, where $A$ and $B$ are called the \emph{frame bounds}. One often also writes $\Phi$ for the $N\times M$ matrix whose $i$th column is the vector $\varphi_i$. When $A=B$, the frame is called a {\it tight frame}. Furthermore, $A=B=1$ produces a {\it Parseval frame}. In the sequel, the set of frames  with $M$ vectors in $\R^N$ will be denoted by $\fr(M,N)$. We refer to \cite{co03} for an introduction to frame theory and to \cite{ck12}  for an overview of the current research in the field.

A frame $\Phi = \{\varphi_i\}_{i=1}^M$ for $\R^N$ is called ({\em strictly}) \emph{scalable} if there exist nonnegative (positive, respectively) scalars $\{s_i\}_{i=1}^M$ such that $\{s_i\varphi_i\}_{i=1}^M$ is a tight frame for $\R^N$. The set of (strictly) scalable frames is denoted by $\sca(M,N)$ ($\sca_{+}(M,N)$, respectively).
This definition obviously allows one to restrict the study to the class of unit norm frames
$$
\Fc_u(M,N) := \left\{\{\varphi_i\}_{i=1}^M\in\Fc(M,N) : \|\varphi_i\|_2=1\text{ for }i=1,\ldots,M\right\},
$$
and further to substitute tight frame by Parseval frame in the above definition. Therefore  a frame $\Phi=\{\varphi_i\}_{i=1}^M$
is scalable if and only if there exist non-negative scalars $\{c_i\}_{i=1}^M$ such that
\begin{equation}\label{equ:scalar}
\Phi C \Phi^T=\sum_{i=1}^Mc_i\varphi_1\varphi_i^T=I, \quad\text{where } C=\text{diag}(c_i).
\end{equation}

In \cite{kopt13},
characterizations of $\sca(M,N)$ and $\sca_{+}(M,N)$,
both of functional analytic and geometric type were derived in the infinite as well as finite dimensional
setting. As a sequel, using topological considerations, it was proved in \cite{kop13c} that the set
of scalable frames, $\sca(M,N)$, is  a `small'  subset of $\fr(M,N)$ when $M$ is relatively small and a
yet different characterization using a particular mapping was derived. This last mapping is closely
related to the so-called diagram vectors/mapping in \cite{cklmns12}. In \cite{cc12}, arbitrary scalars
in $\C$ were allowed, and it was shown that in this case most frames are either not scalable or scalable in
a unique way and, if uniqueness is not given, the set of all possible sequences of scalars is studied.

\subsection{How Scalable is a Frame?}
However, in the applied world, scalability seems too idealistic, in particular, if our frame at hand is
not scalable. This calls for a measure of `closeness to being scalable'. It is though not obvious how
to define such a measure, and one can easily justify different points of view of what `closeness' shall
mean. Let us discuss the following three viewpoints:
\begin{itemize}
\item {\em Distance to $\sca(M,N)$.} Maybe the most straightforward approach is to measure the
distance of a frame $\Phi\in\fr_u(M,N)$ to the set of scalable frames:
\[
d_{\Phi} := \inf_{\Psi\in\sca(M,N)}\|\Phi-\Psi\|_F.
\]
This notion seems natural if we anticipate efficient algorithmic approaches for computing the closest
scalable frame by projections onto $\sca(M,N)$.

\item {\em Conical Viewpoint.} Inspired by \eqref{equ:scalar}, we observe that $\Phi$ is scalable if and only if the identity operator $I$ lies in the cone generated by the vectors $\varphi_i\varphi_i^T$, $i=1,\ldots,M$, which is $\{\sum_{i=1}^M c_i\varphi_i\varphi_i^T:c_i\geq0\}$. Thus the distance of $I$ to this cone seems to be another suitable measure for scalability of $\Phi\in\fr_u(M,N)$, and we define
\[
D_{\Phi} := \min_{C\ge 0\text{ diagonal}}\left\|\Phi C\Phi^T - I\right\|_F,
\]
where $\|\cdot\|_F$ denotes the Frobenius norm. Note that the minimum is attained because this polyhedral cone is closed. This conical viewpoint leads to a computationally efficient algorithm, since we can recast the problem as a quadratic program (see Section \ref{sec:complexity}).

\item {\em Ellipsoidal Viewpoint.} Finally, one can consider the ellipsoid of minimal volume (also known as the
L\"{o}wner ellipsoid) circumscribing the convex hull of the symmetrized frame of $\Phi\in\fr_u(M,N)$:
\[
\Phi_{\rm Sym} := \{\varphi_i\}_{i=1}^M\cup\{-\varphi_i\}_{i=1}^M,
\]
which in the sequel we denote by $E_{\Phi}$ and refer to as the \emph{minimal ellipsoid  of $\Phi$}. Its `normalized' volume is defined by
\[
V_{\Phi} := \frac{\Vol(E_{\Phi})}{\omega_N},
\]
where $\omega_N$ is the volume of the unit ball in $\R^N$. By definition, we have $V_\Phi \le 1$, and  we will later show (Theorem \ref{pro:iff}) that $V_{\Phi}=1$ holds if and only if the frame $\Phi$ is scalable. Hence, yet another conceivably useful measure for
scalability is the closeness of $V_{\Phi}$ to $1$.  This ellipsoidal viewpoint establishes a novel link to convex geometry. Moreover, it will turn out that this measure is of
particular use when estimating the probability of a random frame being scalable.
\end{itemize}
Each notion seems justified from a different perspective, and hence there is no `general truth' for
what the best measure is.

\subsection{Our Contributions}
Our contributions are three-fold: First, we introduce the scalability measures $d_\Phi$, $D_{\Phi}$, and $V_{\Phi}$,
derive estimates for their values, and study their relations in Theorems~\ref{thm_VD} and \ref{thm:comp}. Second, with
Theorems \ref{pro:iff} and \ref{pro:ns} we provide new necessary and sufficient conditions for scalability based on the ellipsoidal viewpoint.
And, third, we estimate the probability of a frame being scalable when each frame vector is drawn independently
and uniformly from the unit sphere (see Theorem \ref{thm:pro}).

\subsection{Expected Impact}
We anticipate our results to have the following impacts:

\begin{itemize}
\item \emph{Constructions of Scalable Frames:} One construction procedure which is a byproduct of our analysis is to
consider random frames with the probability of scalability being explicitly given. However, certainly, there
is the need for more sophisticated efficient algorithmic approaches. But with the measures provided in our
work, the groundwork is laid for analyzing their accuracy.
\item \emph{Extensions of Scalability:} One might also imagine other methodological approaches to modify a frame to
become tight. If sparse approximation properties is what one seeks, another possibility is to be allowed to
take linear combinations of `few' frame vectors in the spirit of the `double sparsity' approach in \cite{RZE10}.
The introduced three quality measures provide an understanding of scalability which we hope might allow an
attack on analyzing those approaches as well.
\item \emph{$\epsilon$-Scalability:} One key question even more important to applications than scalability is that of what is typically loosely coined $\epsilon$-scalability, meaning a frame which is scalable `up to an $\epsilon$',
but which was not precisely defined before. The scalability measures now immediately provide even three definitions of $\epsilon$-scalability in a very natural way, opening three doors to approaching this problem.
\item \emph{Convex Geometry:} The ellipsoidal viewpoint of scalability provides a very interesting link between frame theory and convex geometry. Theorem \ref{pro:ns} and Theorem \ref{thm:pro} are results about frames using convex geometry tools; Theorem \ref{thm:cha} is a result about minimal ellipsoids exploiting frame theory.  We strongly expect the link  established in this paper to bear further fruits in frame theory, in particular the approach of regarding frames from a convex geometric viewpoint by analyzing the convex hull of a (symmetrized) frame.
\end{itemize}

\subsection{Outline}
This paper is organized as follows. In Section, \ref{sec:measures}, the three measures of
closeness of a given frame to be scalable are introduced in three respective subsections
and some basic properties are studied. This is followed by a comparison of the measures
both theoretically and numerically (Section \ref{sec:comp}). Finally, in Section \ref{sec:prob}
we exploit those results to analyze the probability of a frame to be scalable. Interestingly,
along the way we derive necessary and sufficient (deterministic) conditions for a frame
to be scalable (see Subsection \ref{subsec:Necsuffconditions}).

\vspace{.5cm}
\section{Properties of the measures of Scalability}\label{sec:measures}
In this section, we explore some basic properties of the three measures of scalibility which we introduced in the previous section. As mentioned before, we consider only unit norm frames.

\subsection{Distance to the Set of Scalable Frames}
Recall that the measure $d_\Phi$ was defined as the distance of $\Phi$ to the set of scalable frames:
\begin{equation}\label{eq:dist}
d_{\Phi} = \inf_{\Psi\in\sca(M,N)}\|\Phi-\Psi\|_F.
\end{equation}
Since the set $\sca(M,N)$ is not closed (choose $\Phi\in\sca(M,N)$, then $(\frac 1 n\Phi)_{n\in\N}$ is a sequence in $\sca(M,N)$ which converges to the zero matrix), it is not clear whether the infimum in \eqref{eq:dist} is attained. The following proposition, however, shows that this is the case if $d_\Phi < 1$.

\begin{proposition}\label{p:closed}
If $\Phi\in\Fc_u(M,N)$ such that $d_\Phi<1$ then there exists $\hat{\Phi}\in\sca(M,N)$ such that $\|\Phi-\hat{\Phi}\|_F = d_\Phi$.
\end{proposition}
\begin{proof}
Let $\varepsilon = \frac{1-d_\Phi}{2}$, and $\{\Phi_n\}_{n\in\N}\subset\sca(M,N)$ be a sequence of scalable frames such that $\|\Phi-\Phi_n\|_F \leq d_\Phi+\varepsilon/n $. The sequence $\{\Phi_n\}_{n\in\N}$ is bounded as
$$
\|\Phi_n\|_F\leq\|\Phi\|_F+\|\Phi-\Phi_n\|_F\leq\sqrt{M}+d_{\Phi}+\frac{1-d_\Phi}{2},
$$
so without loss of generality, we assume that $\{\Phi_n\}_{n\in\N}$ converges to some $\hat\Phi\in\R^{N\times M}$. It remains to prove that $\hat{\Phi}$ is scalable. For this, denote by $\varphi_{i,n}$ the $i$-th column of $\Phi_n$. Then
$$
\|\varphi_{i,n}\|_2\ge\|\varphi_i\|_2-\|\varphi_i-\varphi_{i,n}\|_2\geq 1-d_\Phi-\varepsilon = \varepsilon
$$
for all $n\ge 0$ and all $i\in\{1,\ldots,M\}$.  Let $C_n = \diag(c_{1,n},\ldots,c_{M,n})$ be a non-negative diagonal matrix such that $\Phi_nC_n\Phi_n^T = I$. Now, for each $j\in\{1,\ldots,M\}$ and each $n\ge 0$ we have
$$
N = \tr(I) = \tr\left(\sum_{i=1}^Mc_{i,n}\varphi_{i,n}\varphi_{i,n}^T\right) = \sum_{i=1}^Mc_{i,n}\|\varphi_{i,n}\|^2\,\ge\,\varepsilon^2 c_{j,n}.
$$
Therefore, each sequence $(c_{i,n})_{n\in\N}$ is bounded. Thus, we find an index sequence $(n_k)_{k\in\N}$ such that
$$
c_i := \lim_{k\to\infty}\,c_{i,n_k}
$$
exists for each $i\in\{1,\ldots,M\}$. Now, it is easy to see that $\Phi_{n_k}C_{n_k}\Phi_{n_k}^T$ converges to $\hat{\Phi} C\hat{\Phi}^T$ as $k\to\infty$, where $C := \diag(c_1,\ldots,c_m)$. Hence, $\hat{\Phi} C\hat{\Phi}^T = I$, and $\hat{\Phi}$ is a scalable frame.
\end{proof}

\begin{remark}\label{rem:non-zero}
The proof of Proposition \ref{p:closed} also yields that the frame vectors of any minimizer of \eqref{eq:dist} are non-zero if $d_\Phi < 1$.
\end{remark}

\begin{lemma}\label{d}
Assume that $d_\Phi < 1$, and let $\hat\Phi = \{\hat\vphi_i\}_{i=1}^M$ be a minimizer of \eqref{eq:dist}. Then for every $i = 1,\ldots,M$,
\begin{enumerate}
\item[(i)]
$\label{equ:hatphi2}
\<\vphi_i,\hat\vphi_i\> = \|\hat\vphi_i\|_2^2.$
\item[(ii)]  $\|\hat\vphi_i\|_2\leq 1$, and equality holds if and only if $\hat\varphi_i=\varphi_i$.
\item[(iii)] $\|\hat\Phi\|_F^2 = M - d_\Phi^2$.
\end{enumerate}
\end{lemma}
\begin{proof}
(i). Fix $j\in\{1,\ldots,M\}$ and $\alpha\in\R\setminus\{0\}$ be arbitrary. Define the frame $\Psi = \{\psi_i\}_{i=1}^M$ as
$$
\psi_i =
\begin{cases}
\hat\vphi_i&\text{if }i\neq j\\
\alpha\hat\vphi_j&\text{if }i=j
\end{cases},
$$
which is scalable. Hence, we have
\begin{align*}
\|\Phi - \hat\Phi\|_F^2
&\le \|\Phi - \Psi\|_F^2 = \sum_{i=1}^M\|\vphi_i - \psi_i\|_2^2 = \sum_{i\neq j}^M\|\vphi_i - \hat\vphi_i\|_2^2 + \|\vphi_j - \alpha\hat\vphi_j\|_2^2\\
&= \sum_{i=1}^M\|\vphi_i - \hat\vphi_i\|_2^2 + \left(\|\vphi_j - \alpha\hat\vphi_j\|_2^2 - \|\vphi_j - \hat\vphi_j\|_2^2\right)\\
&= \|\Phi - \hat\Phi\|_F^2 + \left(\|\vphi_j - \alpha\hat\vphi_j\|_2^2 - \|\vphi_j - \hat\vphi_j\|_2^2\right).
\end{align*}
This implies
$$
\|\vphi_j - \alpha\hat\vphi_j\|_2^2\,\ge\,\|\vphi_j - \hat\vphi_j\|_2^2
$$
or, equivalently,
\begin{equation}\label{strange_alpha}
-2\alpha\left\<\vphi_j,\hat\vphi_j\right\> + \alpha^2\|\hat\vphi_j\|_2^2\,\ge\,-2\left\<\vphi_j,\hat\vphi_j\right\> + \|\hat\vphi_j\|_2^2
\end{equation}
for all $\alpha\in\R\setminus\{0\}$ and all $j\in\{1,\ldots,M\}$. Putting $\alpha = \tfrac{\<\vphi_j,\hat\vphi_j\>}{\|\hat\vphi_j\|_2^2}$ in \eqref{strange_alpha} gives
$$
-\frac{\<\vphi_j,\hat\vphi_j\>^2}{\|\hat\vphi_j\|_2^2}\,\ge\,-2\<\vphi_j,\hat\vphi_j\> + \|\hat\vphi_j\|_2^2,
$$
which is equivalent to
$$
0\,\ge\,\left(\frac{\<\vphi_j,\hat\vphi_j\>}{\|\hat\vphi_j\|_2} - \|\hat\vphi_j\|_2\right)^2,
$$
which leads to the conclusion.

(ii). By (i) we have
$$
\|\hat\vphi_j\|_2^2 = \<\vphi_j,\hat\vphi_j\>\,\le\,\|\vphi_j\|_2\|\hat\vphi_j\|_2 = \|\hat\vphi_j\|_2
$$
This proves $\|\hat\vphi_j\|_2\le 1$ and that $\|\hat\vphi_j\|_2 = 1$ holds if and only if $\hat\vphi_j = \lambda\vphi_j$ for some $\lambda\in\R$. In the latter case, as both vectors are normalized, we have $\lambda = \pm 1$. But $\hat\vphi_j = -\vphi_j$ is impossible due to (i). Thus, $\vphi_j = \hat\vphi_j$ follows.

(iii). By (i),
\begin{align*}
M - d_\Phi^2
&= M - \|\Phi - \hat\Phi\|_F^2 = M - \sum_{i=1}^M\|\vphi_i - \hat\vphi_i\|_2^2\\
&= M - \sum_{i=1}^M\left(1 - 2\<\vphi_i,\hat\vphi_i\> + \|\hat\vphi_i\|_2^2\right) = \sum_{i=1}^M\|\hat\vphi_i\|_2^2 = \|\hat\Phi\|_F^2.
\end{align*}
This proves the claim.
\end{proof}

Since we do not yet have a complete understanding of the set $\sca(M,N)$, we do not have an algorithm for calculating the infimum $d_\Phi$ in \eqref{eq:dist}. For this reason, we introduce two other measures of scalability in the remainder of this section which are more accessible in practice.  We will relate these measures to each other and to $d_\Phi$ in Section \ref{sec:comp}.

\subsection{Distance of the Identity to a Cone}

As mentioned in the introduction, the measure $D_\Phi$ for the scalability of $\Phi\in\fr_u(M,N)$ is the distance of the identity operator on $\R^N$ to the cone generated by $\{\vphi_i\vphi_i^T\}$. Let us recall its definition:
\begin{equation}\label{equ:proj}
D_{\Phi} := \min_{c_i\geq0}\left\|\sum_{i=1}^M c_i\varphi_i\varphi_i^T-I\right\|_F = \min_{C\ge 0\text{ diagonal}}\left\|\Phi C\Phi^T - I\right\|_F.
\end{equation}
For the following, it is convenient to represent the function to be minimized in \eqref{equ:proj} in another form:
\begin{align}\notag
\left\|\sum_{i=1}^Mc_i\vphi_i\vphi_i^T - I\right\|_F^2
&= \tr\left(\sum_{i,j=1}^Mc_ic_j\vphi_i\vphi_i^T\vphi_j\vphi_j^T - 2\sum_{i=1}^Mc_i\vphi_i\vphi_i^T + I\right)\\
\label{equ:quadratic}&= \sum_{i,j=1}^Mc_ic_j|\<\vphi_i,\vphi_j\>|^2 - 2\sum_{i=1}^Mc_i + N.
\end{align}
If we now put ${\mathbf 1} := (1,\ldots,1)^T\in\R^M$, $f_{ij} := |\<\vphi_i,\vphi_j\>|^2$, $i,j=1,\ldots,M$, $F := (f_{ij})_{i,j=1}^M$, and $c := (c_1,\ldots,c_m)^T$, we obtain
\begin{equation}\label{g}
g(c):=\left\|\sum_{i=1}^Mc_i\vphi_i\vphi_i^T - I\right\|_F^2 = c^TFc - 2\cdot \one^Tc + N.
\end{equation}
First of all, we can associate $D_{\Phi}$ with the frame potential (see, e.g., \cite{BF}):
$$
\FP(\Phi) := \sum_{i,j=1}^M|\<\vphi_i,\vphi_j\>|^2.
$$
By plugging in $\alpha{\mathbf 1}$ into $g$ with $\alpha > 0$:
$$
g(\alpha{\mathbf 1}) =  \alpha^2\FP(\Phi) - 2M\alpha + N.
$$
So,
$$
D_\Phi^2\,\le\,\min_{\alpha\ge 0}g(\alpha\one) =  N - \frac{M^2}{\FP(\Phi)}.
$$
We summarize the above discussion in a proposition.

\begin{proposition}
For $\Phi\in\Fc_u(M,N)$ we have
\begin{equation}\label{upbound}
D_\Phi^2\,\le\,N - \frac{M^2}{\FP(\Phi)}.
\end{equation}
\end{proposition}

\begin{remark}
Since $\FP(\Phi)<M^2$, the inequality \eqref{upbound} implies that $D_{\Phi}< \sqrt{N-1}$. It is worth noting that this upper bound is sharp in the sense that for each $\veps > 0$ there exists $\Phi\in\Fc_u(M,N)$ such that $D_\Phi > \sqrt{N-1} - \veps$. This can be proved by essentially choosing the frame vectors of $\Phi$ very close to each other.
\end{remark}

The following proposition can be thought of as an analog to Lemma \ref{d} (iii).
\begin{proposition}\label{trace}
Let the non-negative diagonal matrix $\hat C = \diag(\hat c_1,\ldots,\hat c_M)\in\R^{M\times M}$ be a minimizer of \eqref{equ:proj}. Then
\begin{equation}\label{equ:ci}
\tr(\Phi \hat C\Phi^T)=\sum_{i=1}^M \hat c_i=N-D_{\Phi}^2.
\end{equation}
\end{proposition}
\begin{proof}
The first equality in \eqref{equ:ci} is due to the fact that the $\vphi_i$'s are normalized. Define
$$
f(\alpha) :=  g(\alpha \hat c) = \alpha^2 \hat c^TF\hat c - 2\alpha\one^T\hat c + N.
$$
for $\alpha>0$. The function $f(\alpha)$ has a local minimum at $\alpha=1$, therefore
$$
\frac{df}{d\alpha}\Big|_{\alpha=1}=0\quad\Longrightarrow\quad \hat c^TF\hat c = \one^Tc.
$$
So,
\[
D_{\Phi}^2 = f(1) =\hat  c^TF\hat c - 2\cdot\one^T\hat c + N = N - \one^T\hat c = N - \sum_{i=1}^M\hat c_i,
\]
which proves the proposition.
\end{proof}

\subsection{Volume of the Smallest Ellipsoid Enclosing the Symmetrized Frame}\label{sec:v}
In the following, we shall examine the properties of the measure $V_\Phi$. We will have to recall a few facts from convex geometry, especially results dealing with the ellipsoid of a convex polytope first. An $N$-dimensional ellipsoid centered at $c$ is defined as
$$
E(X,c) := c+X^{-1/2}(B) = \{v: \ip{ X(v-c)}{ (v-c)}\leq 1\},
$$
where $X$ is an $N\times N$ positive definite matrix, and $B$ is the unit ball in $\R^N$. It is easy to see that
\begin{equation}\label{equ:vol}
\Vol(E(X,c)) = \det(X^{-1/2})\omega_N.
\end{equation}
Here, as already mentioned in the introduction, $\omega_N$ denotes the volume of the unit ball in $\R^N$.

A {\em convex body} in $\R^N$ is a nonempty compact convex subset of $\R^N$. It is well-known that for any convex body $K$ in $\R^N$ with nonempty interior there is a unique ellipsoid  of minimal volume containing $K$ and a unique ellipsoid  of maximal volume contained in $K$; see, e.g., \cite[Chapter 3]{ntj89}. We refer to \cite{kb97,og10,ntj89} for more on these extremal ellipsoids.

In what follows, we only consider the ellipsoid of minimal volume  that encloses a given convex body, and this ellipsoid will be called the \emph{minimal ellipsoid} of that convex body. The following theorem is a generalization of John's ellipsoid theorem \cite{fj48}, which will be referred  as John's theorem in this paper.

\begin{theorem}\label{thm:fritz}\cite[Theorem 12.9]{og10}
Let $K\subset\R^N$ be a convex body
and let $X$ be an $N\times N$ positive definite matrix. Then the following are equivalent:
\begin{enumerate}
\item[{\rm (i)}]  $E(X,c)$ is the minimal ellipsoid of $K$.
\item[{\rm (ii)}] $K\subset E(X,c)$, and there exist positive multipliers $\{\lambda_i\}_{i=1}^{k}$, and contact points $\{u_i\}_{i=1}^k$ in $K$ such that
\begin{equation}\label{equ:activej}
X^{-1} = \sum_{i=1}^k\lambda_i(u_i-c)(u_i-c)^T,
\end{equation}
\begin{equation}\label{equ:nonsym}
0=\sum_{i=1}^k\lambda_i(u_i-c),
\end{equation}
\begin{equation}\label{equ:last}
u_i\in \partial K\cap \partial E(X,c),\quad i=1,\dots, k.
\end{equation}
\end{enumerate}
\end{theorem}

Given a frame $\Phi = \{\varphi_i\}_{i=1}^M\in\Fc_u(M,N)$, we will apply John's theorem to the convex hull of the symmetrized frame $\Phi_{\rm Sym}=\{\varphi_i\}_{i=1}^M\cup\{-\varphi_i\}_{i=1}^M$.
By $E_{\Phi}$ we will denote the minimal ellipsoid of the convex hull of $\Phi_{\rm Sym}$. We shall also call this ellipsoid the \emph{minimal ellipsoid of $\Phi$}. This is not in conflict with the notion of the minimal ellipsoid of a convex body since the finite set $\Phi$ is not a convex body. The next lemma says that the center of $E_\Phi$ is always 0.

\begin{lemma}\label{l:symmetric}
Let $K$ be a convex body which is symmetric about the origin. Then the center of the minimal ellipsoid of $K$ is 0.
\end{lemma}
\begin{proof}
Let $E(X,c)$ denote the minimal ellipsoid of $K$.
By definition, if $u\in K$ we also have $-u\in K$,  which implies
$$
\<X(u-c),u-c\rangle\,\le\,1
\quad\text{and}\quad
\<X(-u-c),-u-c\rangle \,\le\,1.
$$
Adding those inequalities, we obtain
$$
2\<Xu,u\rangle + 2\<Xc,c\rangle\,\le\,2.
$$
Since $X\in\R^{N\times N}$ is positive definite,  the above equation implies $\<Xu,u\rangle \le 1$ or, equivalently, $u\in E(X,0)$. This proves $K\subset E(X,0)$. And as $E(X,0)$ has the same volume as $E(X,c)$, it follows from the uniqueness of minimal ellipsoids that $c = 0$.
\end{proof}

In the following, we write $E(X)$ instead of $E(X,0)$. For completeness, we now state a version of Theorem~\ref{thm:fritz} that is specifically taylored to our situation.

\begin{corollary}\label{cor:fjframes}
Let $\Phi = \{\varphi_i\}_{i=1}^M\in\fr_u(M,N)$, and let $X$ be an $N\times N$ positive definite matrix. Then the following are equivalent:
\begin{enumerate}
\item[{\rm (i)}]  $E(X)$ is the minimal ellipsoid of $\Phi$.
\item[{\rm (ii)}] There exist nonnegative scalars $\{\rho_i\}_{i=1}^M$ such that
\begin{equation}\label{equ:X}
X^{-1} = \sum_{i=1}^M\,\rho_i\varphi_i\varphi_i^T,
\end{equation}
\begin{equation}\label{equ:inE}
\ip{X\varphi_i}{\varphi_i}\le 1\quad\text{for all }i=1,2,\ldots,M,
\end{equation}
\begin{equation}\label{equ:rhob}
\ip{X\varphi_i}{\varphi_i} = 1\text{ if $\rho_i > 0$}.
\end{equation}
\end{enumerate}
\end{corollary}
\begin{proof}
(i)$\Rightarrow$(ii). By John's theorem, the contact points must be points in the set $\Phi_{\rm Sym}$. Since $\varphi_i\varphi_i^T = (-\varphi_i)(-\varphi_i)^T$, equation \eqref{equ:activej} with the center $c=0$ implies that there exists $I\subset\{1,\ldots,M\}$ such that
$$
X^{-1} = \sum_{i\in I}\,\lambda_i\varphi_i\varphi_i^T.
$$
Setting $\rho_i=\lambda_i$ for $i\in I$ and $\rho_i=0$ for $i\notin I$, we get \eqref{equ:X}. Equation \eqref{equ:inE} follows from the fact that $\vphi_i\in E(X)$ for each $i = 1,\ldots,M$, and equation \eqref{equ:rhob} is implied by \eqref{equ:last}.

(ii)$\Rightarrow$(i). Let $I = \{i : \rho_i>0\}$. Then the assumptions imply conditions \eqref{equ:activej} and \eqref{equ:last}  with $\{u_i\}_{i\in I} = \{\varphi_i\}_{i\in I}$, and $\{\lambda_i\}_{i\in I} = \{\rho_i\}_{i\in I}$. We just need to slightly modify $\{u_i\}, \{\lambda_i\}$ to make it satisfy \eqref{equ:nonsym} as well. Indeed, we replace $u_i$ by the pair $\pm u_i$ each with half the weight of the original $\lambda_i$. Finally, \eqref{equ:inE} implies that the convex hull of $\Phi_{\rm Sym}$ is contained in $E(X)$. Now, (i) follows from the application of John's theorem.
\end{proof}

\begin{remark}\label{rem:activerho}
It is convenient to view \eqref{equ:X} as saying that $\{X^{1/2}\varphi_i\}_{i=1}^M$ is scalable with scalars $\{\sqrt{\rho_i}\}_{i=1}^M$. Therefore by \cite[Remark 3.12]{kop13c} (see also \cite[Corollary 3.4]{cc12}, since the dimension of $\text{span}\{\varphi_i\varphi_i\}_{i=1}^M$ is at most $\tfrac{N(N+1)}{2}$), we can always pick a set of $\rho_i$'s as in (ii) above such that the number of non-zero (i.e., positive) $\rho_i$'s does not exceed $\tfrac{N(N+1)}{2}$.
\end{remark}

Recall that in the introduction we defined a third measure of scalability $V_\Phi$ as follows:
\begin{equation}\label{equ:V}
V_{\Phi} = \frac{\Vol(E_{\Phi})}{\omega_N} = \det\left(X^{-1/2}\right).
\end{equation}
The second equality is due to \eqref{equ:vol}.

Let us now see how $V_\Phi$ relates to scalability of $\Phi$. If $\Phi\in\Fc_u(M,N)$ is scalable then \eqref{equ:X}--\eqref{equ:rhob} hold with $X = I$. Therefore, $E_\Phi = E(I)$ is the unit ball which implies $V_\Phi = 1$. Conversely, if $V_\Phi = 1$ then $E_\Phi$ must be the unit ball since the ellipsoid of minimal volume is unique. Hence, $E_\Phi = E(I)$, and \eqref{equ:X} implies that $\Phi$ is scalable. This quickly provides another characterization of scalability.

\begin{theorem}\label{pro:iff}
A frame $\Phi\in\fr_u(M,N)$ is scalable if and only if its minimal ellipsoid is the $N$-dimensional unit ball, in which case $V_\Phi = 1$.
\end{theorem}

We can now prove an important property of the minimal ellipsoid $E_\Phi$ of a unit norm frame $\Phi$.

\begin{lemma}
Given $\Phi\in\fr_u(M,N)$, let $E(X)$ be the minimal ellipsoid of $\Phi$ where  $X^{-1} = \sum_{i=1}^M \rho_i\varphi_i\varphi_i^T$, and let $\{\lambda_i\}_{i=1}^N$ be the eigenvalues of $X^{-1}$. Then
 \begin{align}
\label{equ:vv}&V_{\Phi} = \prod_{i=1}^N \lambda_i^{1/2},\\
\label{equ:xin}&\tr\left(X^{-1}\right) = \sum_{i=1}^M\rho_i=\sum_{i=1}^N \lambda_i = N.
\end{align}
\end{lemma}
\begin{proof}
The relation \eqref{equ:vv} immediately follows from \eqref{equ:V}. To prove \eqref{equ:xin}, we set $u_i=X^{1/2}\varphi_i$. Then
\begin{equation}\label{equ_tight}
I = X^{1/2}X^{-1}X^{1/2} = X^{1/2}\left(\sum\limits_{i=1}^{M} \rho_i \varphi_i \varphi_i^T\right)X^{1/2} = \sum\limits_{i=1}^{M} \rho_i u_i u_i^T.
\end{equation}
In addition, we know that whenever $\rho_i>0$, we have $\ip{\varphi_i}{X\varphi_i}=1$, or equivalently $\|u_i\|_2=1$. Using this fact as well as \eqref{equ_tight}, we deduce
$$
\sum_{i=1}^M \rho_i = \sum_{\rho_i>0}\rho_i\tr(u_iu_i^T) = \tr\left(\sum_{i=1}^M \rho_i u_iu_i^T\right) = \tr(I) = N.
$$
The lemma is proved.
\end{proof}

Given a frame $\Phi$ with minimal ellipsoid $E_{\Phi} = E(X)$, we have shown in \eqref{equ:xin} that the trace of $X^{-1}$ is always fixed. This naturally raises the question whether any ellipsoid $E(X)$ with $\tr(X^{-1}) = N$ is necessarily  the minimal ellipsoid of some unit norm frame. The next theorem answers this question in the affirmative.

\begin{theorem}\label{thm:cha}
Every ellipsoid $E(X)$ with $\tr(X^{-1}) = N$ is the minimal ellipsoid of some frame $\Phi\in\fr_u(M,N)$.
\end{theorem}
\begin{proof}
Given any invertible positive definite matrix $X^{-1}$ whose trace is $N$, there exists  $\Phi' = \{\varphi_i\}_{i=1}^N \in\Fc_u(N,N)$ such that
\begin{equation}\label{equ:1}
X^{-1}=\sum_{i=1}^N\varphi_i\varphi_i^T.
\end{equation}
This is a direct result of Corollary 3.1 in \cite{cl10}.

Next, we show that $\ip{X\varphi_i}{\varphi_i} = 1$ for all $i=1,\ldots,N$. For this, fix $j\in\{1,\ldots,N\}$ and choose $x\in\{\varphi_i : i\neq j\}^\perp$ with $\ip{x}{\varphi_j} = 1$. Then
$$
1 = \ip{x}{\varphi_j} = \left\<\sum_{i=1}^NX\varphi_i\varphi_i^Tx,\varphi_j\right\> = \ip{X\varphi_j\varphi_j^Tx}{\varphi_j} = \ip{X\varphi_j}{\varphi_j}.
$$
Now, it follows from Corollary \ref{cor:fjframes} that $E(X)$ is the minimal ellipsoid of $\Phi'$. Construct $\Phi\in \fr_u(M,N)$ by adding $M-N$ unit norm vectors inside $E(X)$ to $\Phi'$. Then $E(X)$ is also the minimal ellipsoid of $\Phi$ since \eqref{equ:1} still holds with $N$ replaced by $M$ and $\rho_i = 0$ for $i=N+1,\ldots,M$.
\end{proof}


\begin{remark}
It is possible using the geometric characterization of scalable frames by $V_\Phi$ to define an equivalence relation on $\Fc_u(M,N)$. Indeed, $\Phi,\Psi\in\fr_u(M,N)$ can be defined to be equivalent if $V_{\Phi} = V_{\Psi}$. We denote each equivalence class by the unique volume for all its members. Specifically, for any $0<a\leq 1$, the class $P[M, N, a]$ consists of all $\Phi\in \fr_u(M,N)$ with $V_{\Phi}=a$. Then, $\sca(M,N)=P[M,N,1]$. This also allows a parametrization of $\fr_u(M,N)$:
$$
\fr_u(M,N) = \bigcup_{a\in (0,1]}P[M,N,a].
$$
\end{remark}

\vspace{.5cm}
\section{Comparison of the Measures}\label{sec:comp}
In this section, we relate the three measures $d_\Phi$, $D_\Phi$, and $V_\Phi$ of scalability to each other. Hereby, we will frequently make use of the standard inequalities in the following lemma, in particular the arithmetic geometric means inequality.

\begin{lemma}
Given $a_i > 0$, $i=1,\ldots,N$, we have
\begin{align}
\label{equ:ine}\frac{N}{\sum_{i=1}^N a_i^{-1}}\,\leq\,\prod_{i=1}^Na_i^{\frac{1}{N}}\,\leq\,\frac{\sum_{i=1}^N a_i}{N},\\
\label{equ:prod}\sum_{i<j}a_ia_j\,\geq\,\frac{N(N-1)}{2}\prod_{i=1}^Na_i^{\frac{2}{N}}.
\end{align}
The inequality \eqref{equ:prod} is a special case of the right hand side inequality of \eqref{equ:ine}.
\end{lemma}

\subsection{Comparison of $D_\Phi$ and $V_\Phi$}
Given a frame $\Phi\in\Fc_u(M,N)$, by definition $V_{\Phi}\leq 1$. Moreover, by Theorem~\ref{pro:iff}, we have $V_{\Phi}=1$ if and only if the frame is scalable.  Intuitively, when a frame is scalable, the frame vectors spread out in the space, which makes its minimal ellipsoid to be the unit ball. But when a frame gets more and more non-scalable, the frame vectors tend to bundle in one place, and thus produce a very ``flat'' ellipsoid with small volume. In this section, we formalize this intuition, and establish that $V_{\Phi}$ is just as suitable as $D_{\Phi}$ in quantifying how scalable a frame is.

We first consider the 2-dimensional case, where there is a straightforward characterization of scalability:  $\Phi=\{\varphi_i\}_{i=1}^M$ is a scalable frame of $\R^2$ if and only if the smallest double cone (with apex at origin) containing all the frame vectors of $\Phi_{\rm Sym}$ has an apex angle greater than or equal to $\pi/2$. This is essentially proved in \cite[Corollary 3.8]{kopt13}; See also Remark \ref{rem:N2} (b).

\begin{example}\label{ex:2dim}
Given $\Phi\in\Fc_u(M,2)$, suppose $\varphi_1,\varphi_2\in\Phi_{\rm Sym}$ generate the smallest cone containing $\Phi_{\rm Sym}$. Without loss of generality, we assume $\vphi_1 = (\cos\theta, \sin\theta)$ and $\vphi_2 = (\cos\theta, -\sin\theta)$, where $2\theta$ is the apex angle. We have $E_\Phi = E_{\{\vphi_1,\vphi_2\}}$,
and this ellipsoid is determined by the solution of the following problem:
\[
\min_{a,b} ab \qquad \text{s.t.} \qquad \frac{\cos^2\theta}{a^2}+\frac{\sin^2\theta}{b^2}=1.
\]
The solution is $a=\sqrt{2}\cos \theta $, $b=\sqrt 2 \sin \theta.$ So in this case,
$$
X^{-1}=\left(\begin{matrix} 2\cos^2\theta & 0\\ 0 & 2\sin^2 \theta \end{matrix}\right)=\varphi_1\varphi_1^T+\varphi_2\varphi_2^T,
$$
and $V_{\Phi} = \det(X^{-1/2}) = \sin 2\theta$.

Now let us calculate $D_{\Phi}$. Since all vectors of $\Phi_{\rm Sym}$ are contained in the cone $\{\pm(a\varphi_1+b\varphi_2), a,b\geq0\}$, any $\varphi_i$ can be represented as $\varphi_i = c\varphi_1 + d\varphi_2$ with $cd\ge 0$. Thus $\varphi_i\varphi_i^T=c^2\varphi_1\varphi_1^T+d^2\varphi_2\varphi_2^T+cd(\varphi_1 \varphi_2^T+\varphi_2\varphi_1^T)$. Therefore,  the Frobenius norm minimization problem becomes
\[
\min\limits_{a,b,c\geq0}\left\|a\varphi_1\varphi_1^T+b\varphi_2\varphi_2^T+c(\varphi_1 \varphi_2^T+\varphi_2\varphi_1^T)-I\right\|_F.
\]
The solution of this problem is $a=b=\frac{2}{3 + \cos4\theta}$, $c=0$, and thus
$$
D_\Phi^2 = 2-2a = 2-\frac{2}{2-V_{\Phi}^2}.
$$
So, as $V_{\Phi}$ is approaching 1, $D_{\Phi}$ is approaching 0, and vice versa.
\end{example}

In Example \ref{ex:2dim} it is shown that in the 2-dimensional case, $V_{\Phi}$ is a function of $D_{\Phi}$. However, in general $V_{\Phi}$ is no longer uniquely determined by $D_{\Phi}$ but falls into a range defined by $D_{\Phi}$ as the following theorem indicates. But the key point here is that it still remains true that $D_{\Phi}$ approaches zero if and only if the volume ratio tends to one.

\begin{theorem}\label{thm_VD}
Let $\Phi=\{\varphi_i\}_{i=1}^M\in \fr_u(M,N)$, then
\begin{equation}\label{equ:VD}
\frac{N(1-D_{\Phi}^2)}{N-D_{\Phi}^2}\,\leq\,V_{\Phi}^{4/N}\,\leq\,\frac{N(N-1-D_\Phi^2)}{(N-1)(N-D_\Phi^2)}\,\le\,1,
\end{equation}
where the leftmost inequality requires $D_{\Phi}<1$. Consequently, $V_{\Phi}\rightarrow 1$ is equivalent to $D_{\Phi}\rightarrow 0$.
\end{theorem}
\begin{proof}
The rightmost inequality is clear. Let us prove the upper bound on $V_\Phi^{4/N}$ in \eqref{equ:VD}. For this, let $E_\Phi = E(X)$ be the minimal ellipsoid of $\Phi$, and let $\{\lambda_i\}_{i=1}^N$ be the eigenvalues of $X^{-1}  = \sum_{i=1}^M\rho_i\varphi_i\varphi_i^T$. For any $\alpha>0$, we have
$$
D_{\Phi}^2\leq\left\|\sum_{i=1}^M\alpha\rho_i\varphi_i\varphi_i^T - I\right\|_F^2=  \|\alpha X^{-1}-I\|_F^2 = \sum\limits_{i=1}^{N} (\alpha\lambda_i-1)^2 = \alpha^2\sum\limits_{i=1}^{N}\lambda_i^2 - 2\alpha\sum \limits_{i=1}^{N}\lambda_i + N.
$$
Therefore, by \eqref{equ:xin},\begin{equation}\label{equ:Dphi}
D_{\Phi}^2  \leq\min_{\alpha>0}\left(\alpha^2\sum\limits_{i=1}^{N}\lambda_i^2 - 2\alpha\sum \limits_{i=1}^{N}\lambda_i + N\right)=N-\frac{N^2}{\sum_{i=1}^N\lambda_i^2}.
\end{equation}
We use \eqref{equ:vv} and \eqref{equ:prod} to estimate $\sum_{i=1}^N\lambda_i^2$:
\begin{align}\notag
\sum\limits_{i=1}^{N}\lambda_i^2&=\left(\sum \limits_{i=1}^{N}\lambda_i\right)^2 - 2\sum\limits_{i<j}\lambda_i\lambda_j=N^2-2\sum\limits_{i<j}\lambda_i\lambda_j\\
\label{equ:lambda}&\leq N^2-N(N-1)\prod_{i=1}^{N}\lambda_i^{2/N}=N^2-N(N-1)V_\Phi^{4/N}.
\end{align}
Plugging \eqref{equ:lambda} in \eqref{equ:Dphi} and solving it for $V_\Phi^{4/N}$ yields the upper bound in \eqref{equ:VD}.

For the lower bound, let $\hat C = \text{diag}\{c_i\}_{i=1}^{M}$ be a minimizer of \eqref{equ:proj}. Then $D_{\Phi}=\|\Phi \hat C \Phi^T-I\|_F$. Moreover,
\begin{align*}
\tr(\Phi \hat C\Phi^TX)=\sum\limits_{i=1}^{M} c_i \varphi_i^T X \varphi_i\leq \sum\limits_{i=1}^{M} c_i.
\end{align*}
The last inequality holds due to \eqref{equ:inE}. Therefore,
\begin{align}
\begin{split}\label{equ:vd1}
\tr(X)
&=\tr(\Phi \hat C\Phi^T X) -\tr((\Phi \hat C \Phi^T-I)X)\\
&= \tr(\Phi \hat C\Phi^T X) - \tr(\Phi \hat C\Phi^T-I) - \tr((\Phi \hat C\Phi^T-I)(X-I))\\
&\leq \sum_{i=1}^M c_i - \left(\sum_{i=1}^M c_i - N\right) - \tr((\Phi \hat C\Phi^T-I)(X-I))\\
&\leq  N + \|\Phi \hat C \Phi^T - I\|_F\|X-I\|_F\\
& =  N + D_{\Phi}\sqrt{\sum_{i=1}^N\left(\lambda_i^{-1}-1\right)^2}\\
& =  N + D_{\Phi}\sqrt{\left(\sum_{i=1}^N\lambda_i^{-1}\right)^2 - 2\sum\limits_{i<j}\lambda_i^{-1}\lambda_j^{-1} - 2\sum\limits_{i=1}^{N}\lambda_i^{-1} + N}\\
&\leq N+D_\Phi \sqrt{\tr^2(X)-N(N-1)V_{\Phi}^{-4/N}-2 \tr(X)+N},
\end{split}
\end{align}
where the last inequality is due to \eqref{equ:prod} with $a_i=\lambda_i^{-1}$, and \eqref{equ:vv}. By \eqref{equ:ine},
\begin{equation}\label{equ:trX}
\tr(X) = \sum_{i=1}^{N}\lambda_i^{-1}\geq\frac{N}{\prod\limits_{i=1}^{N}\lambda_i^{1/N}}=NV_{\Phi}^{-2/N}\geq N.
\end{equation}
Now, we subtract $N$ on both sides of \eqref{equ:vd1}, square both sides, and obtain
\[
\left(\tr(X)-N\right)^2\leq D_\Phi^2\left(\tr^2(X) - 2 \tr(X) + N - N(N-1)V_{\Phi}^{-4/N}\right).
\]

The latter inequality is equivalent to
$$
\left(\tr(X) - \frac{N-D_\Phi^2}{1-D_\Phi^2}\right)^2\le\frac{D_\Phi^2(N-1)}{(1-D_\Phi^2)^2}\left(N-D_\Phi^2 - (1-D_\Phi^2)NV_\Phi^{-4/N}\right).
$$
This proves that
$$
N - D_\Phi^2 - (1-D_\Phi^2)NV_\Phi^{-4/N}\,\ge\,0,
$$
which is equivalent to the leftmost inequality in \eqref{equ:VD}.

\end{proof}

\subsection{Algorithms and Numerical Experiments}\label{sec:complexity}
The computation in \eqref{equ:quadratic} shows that $D_{\Phi}$ can be computed via  Quadratic Programming (QP). As is well known, this problem can be solved by many well developed methods, e.g., Active-Set, Conjugate Gradient, Interior point.

The minimal ellipsoid problem has been studied for half a century. For a given convex body $K$ and a small quantity $\eta>0$, a fast algorithm to compute an ellipsoid $E\supseteq K$ with
$$
\Vol(E)\leq (1+\eta)\Vol(\text{Minimal ellipsoid}(K))
$$
is the Khachiyan's barycentric coordinate descent algorithm \cite{kha}, which needs a total of $O(M^{3.5}\ln(M\eta^{-1}))$ operations. For the case $N\ll M$, Kumar and Yildirim \cite{kumar} improved this algorithm using core sets and reduced the complexity to $O(MN^3\eta^{-1})$.

For all numerical simulations in this paper,   we use Khachiyan's method to compute minimal ellipsoids and the active set method to solve the quadratic programming in \eqref{equ:proj}. As expected, we have observed a much faster computational speed of the latter, especially when the problem grows large in size.

Figure \ref{fig:VD} shows the values of $D_{\Phi}$ and $V_{\Phi}$ for randomly generated frames in $\Fc_u(M, 4)$ with $M=6, 11, 15, $ and $20$. In each plot, we generated $1000$ frames, where each column of the frame is chosen uniformly at random from the unit sphere, and calculated both $V_{\Phi}$ and $D_{\Phi}$.

As expected, for a fixed $D_\Phi$, we see a range of $V_{\Phi}$. For a direct comparison, we plotted the two bounds from \eqref{equ:VD}. The lower bound from \eqref{equ:VD} is quite optimal  based on the figure.

On the other hand,  as $M$ increases, we  observe a change of concentration of the points from  scattering around to being heavily distributed around $D_{\Phi}=0$: ``the scalable region''. Indeed, as shown by Theorem~\ref{thm:pro} in Section \ref{sec:prob}, the threshold for having positive probability of scalable frames in dimension $N = 4$ is $N(N+1)/2 = 10$. Therefore, we have considerably many points achieving $D_{\Phi}=0$ for $M=11, 15, 20$. In fact, about $60\%$ of these $1000$ frames in $\fr(4,20)$ are scalable (up to a machine error).

\begin{figure}[ht]
 \centering
 \subfigure{
  \includegraphics[width=0.47\textwidth]{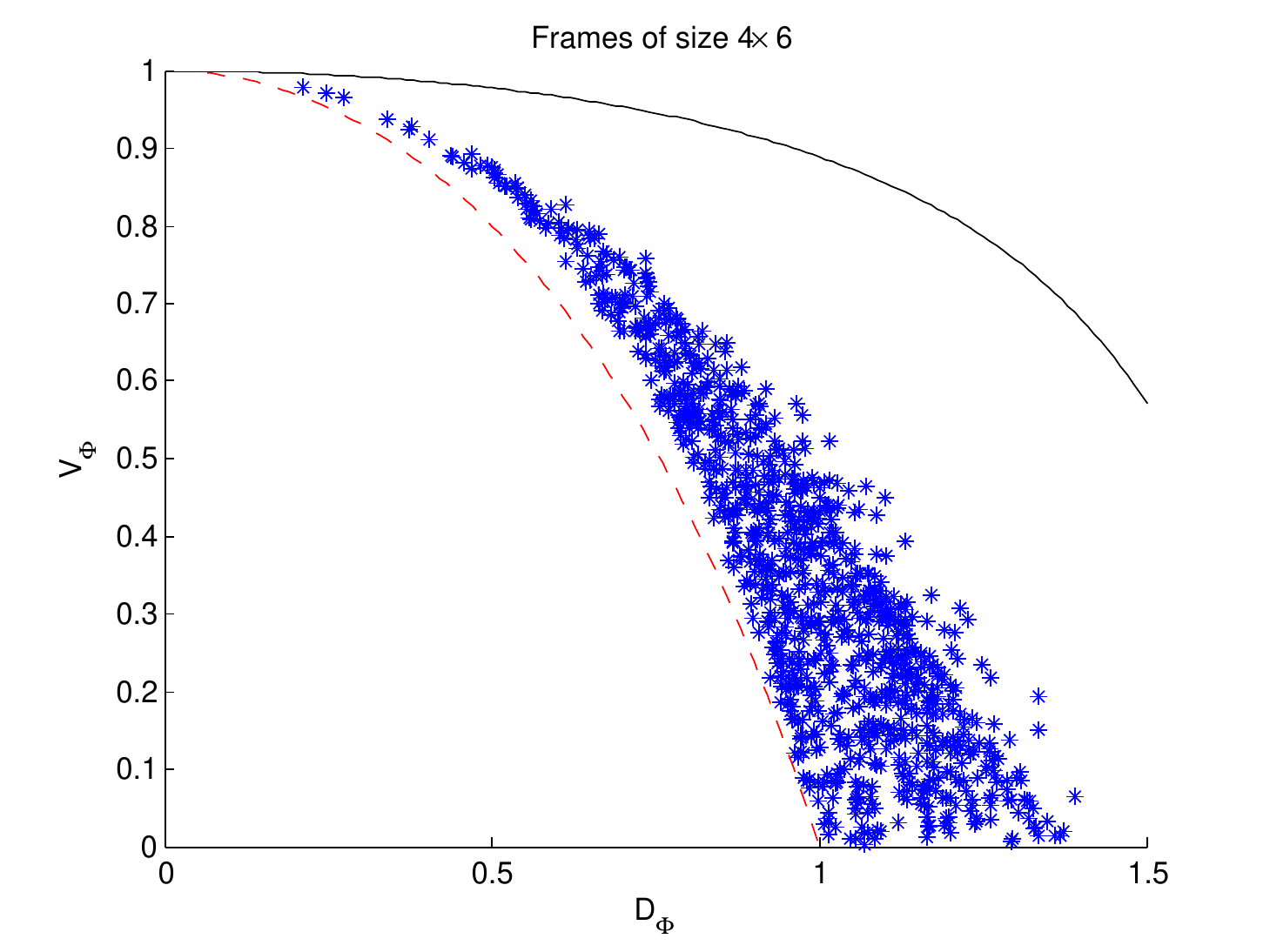}}
 \subfigure
 {
   \includegraphics[width=0.47\textwidth]{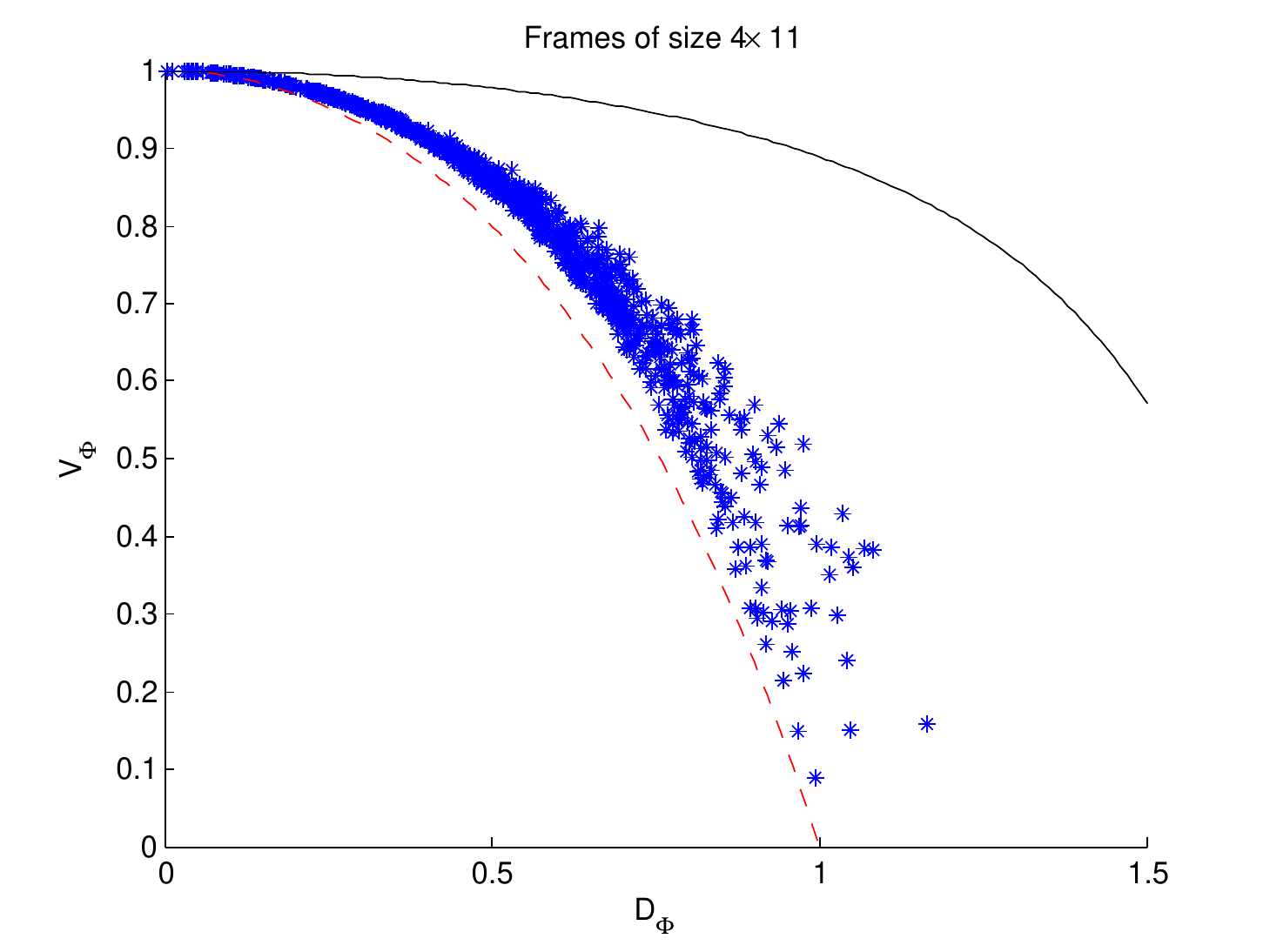}
   }
 \subfigure{
   \includegraphics[width=0.47\textwidth]{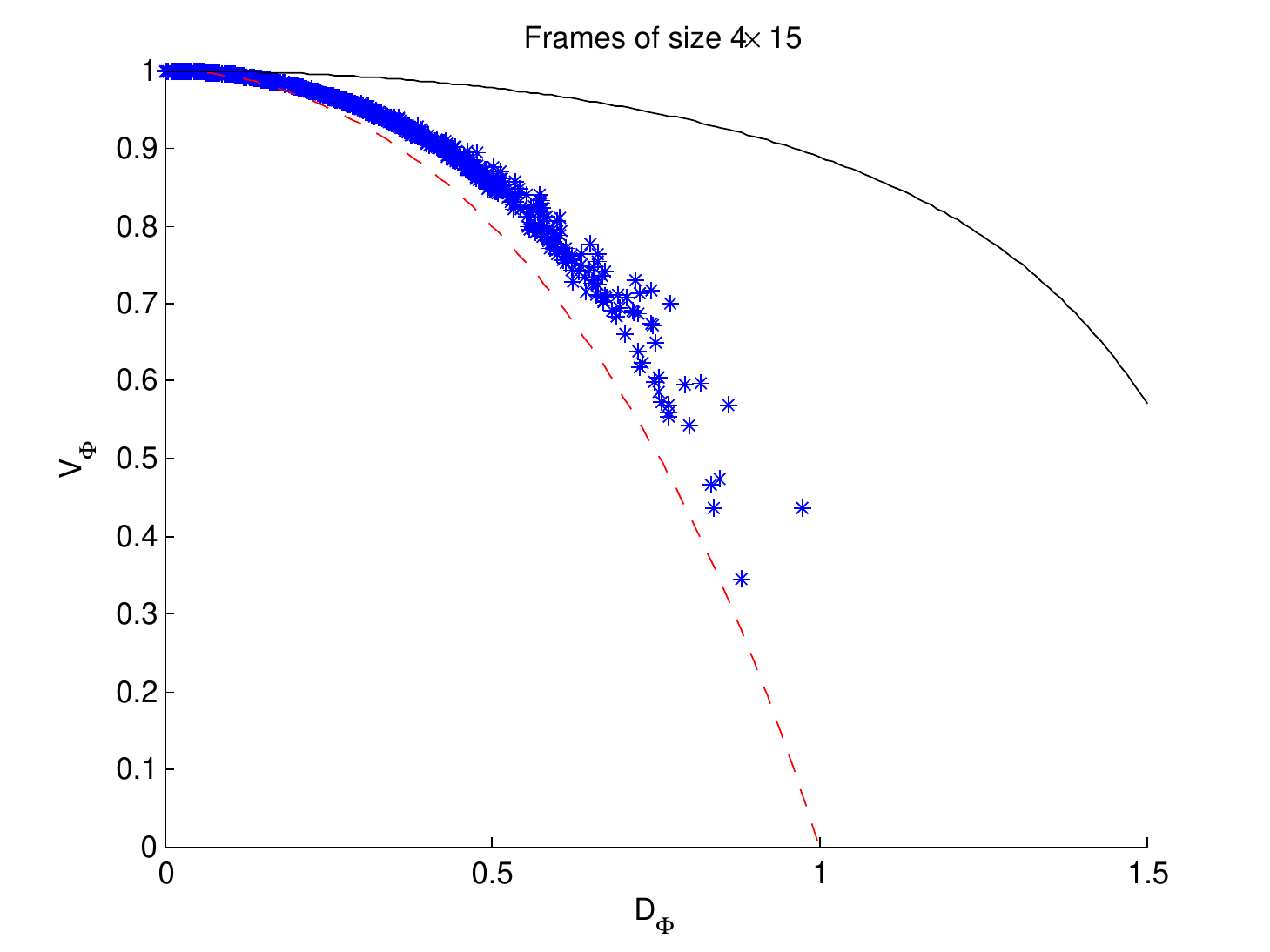}
   }
   \subfigure{
   \includegraphics[width=0.47\textwidth]{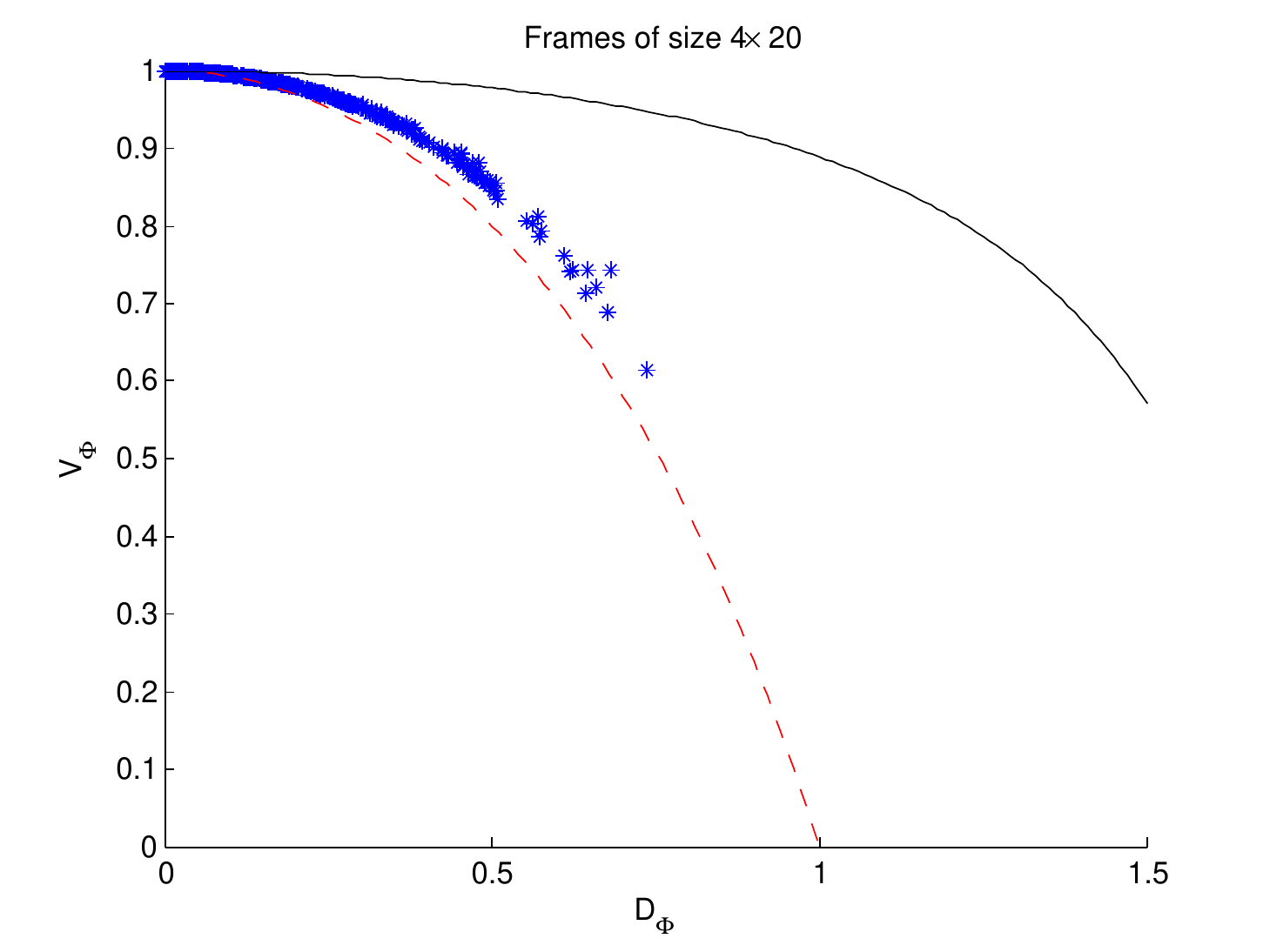}
   }
\caption{\small Relation between $V_{\Phi}$ and $D_{\Phi}$ for $\Phi\in \mathcal{F}_u(M,4)$ with $M = 6,11,15,20$. The solid line indicates the upper bound in \eqref{equ:VD}, while the  dash line indicates the lower bound. }
\label{fig:VD}
\end{figure}

This suggests that the two measures of scalability, the distance between $D_{\Phi}$ and $0$ and the distance between $V_{\Phi}$ and $1$, though closely related, are  indeed different in the sense that there is no one-to-one correspondence between them. An advantage of using $D_\Phi$ to measure scalability lies in the fact  it is more naturally related to the notion of $m-$scalability (defined in \cite{kop13c})  and is more efficient to compute. By contrast, $V_\Phi$ is a more intuitive measure of scalability from a geometric point of view.

\subsection{Comparison of the Measures $D_\Phi$ and $V_\Phi$ with $d_\Phi$}
The distance of a frame $\Phi\in\Fc_u(M,N)$ to the set of scalable frames is the most intuitive and natural measure of scalability. The next theorem shows that the practically more accessible measures $D_\Phi$ and $V_\Phi$ are equivalent to $d_\Phi$ in the sense that $d_\Phi$ tends to zero if and only if the same holds for $D_\Phi$ or $1 - V_\Phi$.

\begin{theorem}\label{thm:comp}
Let $\Phi\in\Fc_u(M,N)$ and assume that $d_{\Phi}<1$. Then with $K := \min\{M,\tfrac{N(N+1)}{2}\}$ and $\omega := D_\Phi + \sqrt{K}$ we have
\begin{equation}\label{equ:Dd}
\frac{D_\Phi}{\omega + \sqrt{\omega^2 - D_\Phi^2}}
\,\leq\, d_\Phi \,\leq\,
\sqrt{KN\left(1 - V_\Phi^{2/N}\right)}.
\end{equation}
Consequently, with the help of Theorem {\rm\ref{thm_VD}}, we can bound $d_{\Phi}$  below and above by  expressions of $D_{\Phi}$ or expressions of $V_{\Phi}$.
\end{theorem}
\begin{proof}
Following the same notation as in the proof of Theorem \ref{thm_VD}, let $\lambda_1,\ldots,\lambda_N$ be the eigenvalues of $X^{-1} = \sum_{i=1}^M\rho_i\vphi_i\vphi_i^T$. Furthermore, let $J = \{i : \rho_i > 0\}$. By Remark \ref{rem:activerho}, $|J|\leq K$. Define a frame $\widetilde{\Phi} = \{\tilde\varphi_i\}_{i=1}^M$ by
\begin{equation}\label{equ:apprscale}
\tilde\varphi_i :=
\begin{cases}
V_\Phi^{1/N}X^{1/2}\vphi_i &\text{if }i\in J\\
\vphi_i        &\text{if }i\notin J.
\end{cases}
\end{equation}
Note that $\widetilde{\Phi}$ is scalable, and, moreover, $\|X^{1/2}\varphi_i\|_2 = 1$ for $i\in J$ by \eqref{equ:rhob}. So,

\begin{align}
\|\Phi - \widetilde\Phi\|_F^2\notag
&=\notag \sum_{i\in J}\|\vphi_i - V_\Phi^{1/N}X^{1/2}\vphi_i\|_2^2 = \sum_{i\in J}\|(X^{-1/2} - V_\Phi^{1/N}I)X^{1/2}\vphi_i\|_2^2\\
&\le \notag\|X^{-1/2} - V_\Phi^{1/N}I\|_F^2\sum_{i\in J}\|X^{1/2}\vphi_i\|_2^2\le K\sum_{j=1}^N\left(\lambda_j^{1/2} - V_\Phi^{1/N}\right)^2\\
&=\notag K\left(N + V_\Phi^{2/N}N - 2V_\Phi^{1/N}\sum_{j=1}^N\lambda_j^{1/2}\right)\\
&=\notag KN\left(1 + V_\Phi^{2/N} - 2V_\Phi^{1/N}\frac 1 N\sum_{j=1}^N\lambda_j^{1/2}\right)\\
&\le\label{middle} KN\left(1 + V_\Phi^{2/N} - 2V_\Phi^{2/N}\right) = KN\left(1 - V_\Phi^{2/N}\right)^2.
\end{align}
As $d_\Phi\le\|\Phi - \widetilde\Phi\|_F$, this proves the right-hand side of \eqref{equ:Dd}.

Let $\hat\Phi$ be a minimizer of \eqref{eq:dist} (which exists due to Proposition \ref{p:closed} and has non-zero columns by Remark \ref{rem:non-zero}). Since $\hat\Phi$ is scalable, there exists a non-negative diagonal matrix $S = \text{diag}(s_i)_{i=1}^M$ such that $\hat\Phi S\hat\Phi^T = I$. Again by Remark \ref{rem:activerho}, we may assume that at most $K$ of the $s_i$ are non-zero.
We then have
$$
\Phi S\Phi^T - I = \Phi S\Phi^T - \hat\Phi S\hat\Phi^T = \sum_{i=1}^Ms_i\left[\vphi_i(\vphi_i^T - \hat\vphi_i^T) + (\vphi_i - \hat\vphi_i)\hat\vphi_i^T\right],
$$
and therefore, as $\|\hat\vphi_i\|_2\le 1$ (see Lemma \ref{d} (ii)),
\begin{align*}
D_\Phi
&\le\|\Phi S\Phi^T - I\|_F\le\sum_{i=1}^Ms_i\big(\|\vphi_i - \hat\vphi_i\|_2 + \|\vphi_i - \hat\vphi_i\|_2\|\hat\vphi_i\|_2\big)\\
&\le 2\sum_{i=1}^Ms_i\|\vphi_i - \hat\vphi_i\|_2\le 2\left(\sum_{i=1}^Ms_i^2\right)^{1/2}\left(\sum_{i=1}^M\|\vphi_i - \hat\vphi_i\|_2^2\right)^{1/2}\\
&\le 2\left(K\max_i s_i^2\right)^{1/2}\|\Phi - \hat\Phi\|_F = 2\sqrt{K}\left(\max_i s_i\right)d_\Phi
\end{align*}
Now, for each $i\in\{1,\ldots,M\}$ we have
$$
s_i
= \frac{\hat{\varphi_i}^Ts_i\hat{\varphi_i}\hat{\varphi_i}^T\hat{\varphi_i}}{\|\hat\varphi_i\|_2^4}
\leq \sum_{k=1}^M\frac{\hat\varphi_i^Ts_k\hat\varphi_k\hat\varphi_k^T\hat\varphi_i}{\|\hat\varphi_i\|_2^4}
= \frac{\hat\varphi_i^T\hat\Phi S\hat\Phi^T\hat\varphi_i}{ \|\hat\varphi_i\|_2^4}
= \frac{1}{ \|\hat\varphi_i\|_2^2}\leq\frac{1}{(1-d_{\Phi})^2},
$$
where the last inequality follows from the triangle inequality. This gives
\begin{equation}\label{for later use}
D_\Phi\,\le\,\frac{2\sqrt{K}d_\Phi}{(1-d_\Phi)^2}.
\end{equation}
Solving for $d_\Phi$ in the last inequality leads to the left hand side of \eqref{equ:Dd}.

\end{proof}

We conclude this section by a theorem on approximating unit norm frames by scalable frames.

\begin{theorem}[Approximation by  scalable frames]\label{thm:app}
Let $\Phi\in\Fc_u(M,N)$ and assume that $d_\Phi\le\frac 1 2(1+\sqrt{K})^{-1}$. Let $\hat\Phi$ be a minimizer of \eqref{eq:dist}, and let $E_\Phi = E(X)$ be the minimal ellipsoid of $\Phi$, where $X^{-1} = \sum_{i=1}^{M}\rho_i\varphi_i\varphi_i^T$. Then the scalable  frame $\widetilde{\Phi} = \{\tilde\varphi_i\}_{i=1}^M$ defined in \eqref{equ:apprscale} is a good approximation to $\Phi$ in the following sense:
\begin{equation}\label{equ:relative}
\|\widetilde{\Phi}-\Phi\|_F
\leq
\sqrt{KN}\left(1 - \sqrt{N\frac{(1 - d_{\Phi})^4 - 4Kd_{\Phi}^2}{N(1 - d_{\Phi})^4 - 4Kd_{\Phi}^2}}\right)^{1/2} = K\sqrt{N}O(d_{\Phi}),
\end{equation}
where $K = \min\{M,\tfrac{N(N+1)}{2}\}$.
\end{theorem}
\begin{proof}
We extend the estimate \eqref{middle} with the help of the leftmost inequality of \eqref{equ:VD}:
\begin{equation}\label{last}
\|\Phi - \widetilde\Phi\|_F^2\,\le\,KN\left(1 - \sqrt{N\frac{1-D_\Phi^2}{N-D_\Phi^2}}\right).
\end{equation}
Since the right-hand side of \eqref{last} is an increasing function of $D_{\Phi}$ on $[0,1]$, we substitute \eqref{for later use} into \eqref{last} and obtain the left hand side of \eqref{equ:relative}, where we need the requirement on $d_{\Phi}$ so that $D_{\Phi}<1$.
\end{proof}

\vspace{.5cm}
\section{Probability of having scalable frames}\label{sec:prob}
This section aims to estimate the probability $P_{M,N}$ of unit norm frames to be scalable when the frame vectors are drawn independently and uniformly from the unit sphere $\Sb^{N-1} \subset \R^N$. This is in a sense equivalent to estimating the ``size'' of $\sca(M,N)$ in $\fr_u(M,N)$.

\smallskip
The basic idea is to use the characterization of scalability in terms of the minimum volume  ellipsoids through John's theorem, see Theorem~\ref{pro:iff}. From this geometric point of view, we derive new and relatively simple conditions for  scalability and non-scalability (Theorem \ref{pro:ns}). These conditions are the key tools we use to  estimate the probability $P_{M,N}.$

\subsection{Necessary and Sufficient Conditions for Scalability}\label{subsec:Necsuffconditions}
The following theorem plays a crucial role in the proof of our main theorem on the probability of having scalable frames in Subsection \ref{ss:prob}. However, it is also of independent interest.

\begin{theorem}\label{pro:ns}
Let $\Phi\in\Fc_{u}(M,N)$. Then the following hold:
\begin{enumerate}
\item[(a)] {\rm (}A necessary condition for scalability\,{\rm )} If $\Phi$ is scalable, then
\begin{equation}\label{e:nec}
\min_{\|d\|_2=1}\max_i |\<d,\varphi_i\>|\geq \frac{1}{\sqrt{N}}.
\end{equation}
\item[(b)] {\rm (}A sufficient condition for scalability\,{\rm )} If
\begin{equation}\label{e:suff}
\min_{\|d\|_2=1}\max_i |\<d,\varphi_i\>|\geq \sqrt{\frac{N-1}{N}},
\end{equation}
then $\Phi$ is scalable.
\end{enumerate}
\end{theorem}
\begin{proof}
(a). We will use the following fact: if $E_K$ is the minimal ellipsoid of a convex body $K\subset\R^N$ which is symmetric about the origin, then $\frac{1}{\sqrt{N}}E_K\subset K$, see \cite[Theorem 12.11]{og10}.
If $\Phi$ is scalable, then the unit ball is the minimal ellipsoid of the convex hull $\co(\Phi_{\rm Sym})$ of $\Phi_{\rm Sym}$. Therefore, $\frac{1}{\sqrt{N}}B\subset\co(\Phi_{\rm Sym})$. And as a continuous convex function on a compact convex set attains its maximum at an extreme point of this set (see, e.g., \cite[Theorem 3.4.7]{np06}), we conclude that for each $d\in \Sb^{N-1}$ we have
$$
\frac 1 {\sqrt{N}} = \max_{x\in\frac{1}{\sqrt{N}}B}|\ip{d}{x}|\leq\max_{x\in\co(\Phi_{\rm Sym})}|\ip{d}{x}|\le\max_i|\ip{d}{\varphi_i}|.
$$

(b). Let $E_{\Phi}=E(X)$ be the minimal ellipsoid of $\Phi$. With a unitary transformation, we can assume $X^{-1/2}=\text{diag}(a_i)_{i=1}^N$. Towards a contradiction, suppose that \eqref{e:suff} holds, but that $\Phi$ is not scalable. Then, by Theorem~\ref{pro:iff}, $a_1\leq a_2\leq\ldots\leq a_N$ with $a_1<a_N$. Take any frame vector $\varphi=(x_1, x_2,\dots,x_N)^T$ from $\Phi$. It satisfies
$\sum_{i=1}^N\frac{x_i^2}{a_i^2} = \<X\vphi,\vphi\>\leq1$ and $\sum_{i=1}^Nx_i^2=1$, which implies
$$
\sum_{i=1}^{N-1}x_i^2\left(\frac{1}{a_i^2}-\frac{1}{a_N^2}\right)\leq1-\frac{1}{a_N^2}.
$$
Hence, setting $\rho=(1-\frac{1}{a_N^2})/(\frac{1}{a_1^2}-\frac{1}{a_N^2})$, we have $x_1^2\le\rho$. We claim that
\begin{equation}\label{equ:rho}
\rho<\frac{N-1}{N}.
\end{equation}
Then we choose $d = (1, 0, \dots,0)^T$ and find that $|\<d,\vphi\>| = |x_1| < \sqrt{\frac{N-1}{N}}$ for each $\vphi\in\Phi$ which contradicts the assumption.

Proving \eqref{equ:rho} is equivalent to proving $\frac{1}{a_N^2}+\frac{N-1}{a_1^2}>N$, which is true because
$$
\frac{1}{a_N^2}+\frac{N-1}{a_1^2}\geq \sum_{i=1}^N\frac{1}{a_i^2} > \frac{N^2}{\sum_{i=1}^N a_i^2}=N,
$$
where we have used \eqref{equ:xin} and \eqref{equ:ine} (in which equality holds if and only if $a_1 = \ldots = a_N$).
\end{proof}

\begin{remark}\label{rem:N2}
(a) Another necessary condition for scalability was proved in \cite[Theorem 3.1]{cklmns12}. We wish to point out that this necessary condition is unrelated to the one given in part (a) of the previous theorem in the sense that neither of these conditions implies the other.

\smallskip
(b) When the dimension $N=2$, Theorem~\ref{pro:ns} gives  a necessary and sufficient condition for a frame to be scalable. This condition can be easily interpreted in terms of cones as already mentioned before: $\{\varphi_i\}_{i=1}^M$ is a scalable frame for $\R^2$ if and only if every double cone with apex at origin and containing $\Phi_{\rm Sym}$ has an apex angle greater than or equal to $\pi/2$.

\smallskip
(c) For a general $N$, the gap between these two conditions is large. However, this gap cannot be improved. Theorem~\ref{pro:ns}(a) is tight in the sense that we cannot replace $1/\sqrt{N}$ by a bigger constant. This is because an orthonormal basis reaches this constant. The sufficient condition is also optimal in the sense that $\sqrt{(N-1)/N}$ cannot be replaced by a smaller number. This requires some more analysis as shown below.
\end{remark}

\begin{proposition}\label{pro:tight}
For any small $\varepsilon>0$ and any $N\in\N$, there exists a unit norm frame $\Phi$ for $\R^N$, such that
$$
\min_{\|d\|_2=1}\max_i|\<d,\varphi_i\>|\geq\sqrt{\frac{N-1}{N}}-2\varepsilon,
$$
but $\Phi$ is not scalable.
\end{proposition}
\begin{proof}
Pick an ellipsoid $E(X)$ with $X^{-1} = \text{diag}(a_1^2,a_2^2,\ldots,a_{N-1}^2,a_N^2)$, where $a_1^2 = a_2^2 = \ldots = a^2_{N-1} = \frac{N-1-\varepsilon}{N-1}$, and $a_N^2=1+\varepsilon$. By Theorem \ref{thm:cha}, there exists a (non-scalable) frame $\Phi\in\Fc_u(M,N)$ whose minimal ellipsoid is $E(X)$.

Then for any $x\in E(X)\cap \Sb^{N-1}$, we have
\[
1\geq \sum\limits_{i=1}^{N-1}\frac{x_i^2}{a_i^2}+\frac{x_N^2}{a_N^2}=\frac{(N-1)(1-x_N^2)}{N-1-\varepsilon}+\frac{x_N^2}{1+\varepsilon},
\]
which implies that
\[
x_N^2\geq \frac{1+\varepsilon}{N}.
\]
Now for any $d=(d_1, d_2, \dots, d_N)\in \Sb^{N-1}$, if $d_N^2<\frac{1+\varepsilon}{N}$, then let
$$
x_0 = \sqrt{1-\frac{1+\varepsilon}{N}}\frac{\tilde{d}}{\|\tilde{d}\|}+\text{sign}(d_N)\left(0,0,...,0,\sqrt{\frac{1+\varepsilon}{N}}\right),
$$
where $\tilde{d}=(d_1,d_2,....,d_{N-1},0)$. It is easy to verify that $x_0\in E(X)\cap \Sb^{N-1}$ and that $\langle x_0, d \rangle\geq \sqrt{\frac{N-1-\varepsilon}{N}}$. If $d_N^2\geq \frac{1+\varepsilon}{N}$, then let $x_0=d$. It is again easy to check $x_0\in E(X)\cap \Sb^{N-1}$ and $\langle x_0,d \rangle =1$. In summary, for any $d\in \Sb^{N-1}$, there exists an $x_0\in E(X)\cap\Sb^{N-1}$, such that  $\langle x_0, d \rangle\geq \sqrt{\frac{N-1-\varepsilon}{N}}$.

We add vectors from the set $E(X)\cap \Sb^{N-1}$ to $\Phi$ such that the frame vectors are dense enough  to form an $\varepsilon$-ball of $E(X)\cap \Sb^{N-1}$, i.e., for any $x\in E(X)\cap \Sb^{N-1}$, there exists a $\varphi_i\in E(X)\cap \Sb^{N-1}$, such that $\|\varphi_i-x\|_2\leq \varepsilon$. Notice this new frame has the same minimal ellipsoid. With this construction, for any $d\in\Sb^{N-1}$, we can find a frame vector $\varphi_i$ such that
$\langle \varphi_{i},d\rangle=\langle x,d\rangle+\langle \varphi_i-x,d\rangle\geq \sqrt{\frac{N-1-\varepsilon}{N}}-\varepsilon \geq \sqrt{\frac{N-1}{N}}-2\varepsilon$ provided that $\varepsilon$ is small enough.
\end{proof}

In Remark \ref{rem:N2}(b), we mentioned that \eqref{e:nec} is necessary and sufficient for scalability if $N = 2$. In the following, we shall show that the same holds if $M = N$:

\begin{theorem}\label{t:M=N}
For $\Phi\in\fr_u(N,N)$, the following statements are equivalent.
\begin{enumerate}
\item[{\rm (i)}]   $\Phi$ is scalable.
\item[{\rm (ii)}]  $\Phi$ is unitary.
\item[{\rm (iii)}] $\min_{\|d\|_2=1}\max_i |\<d,\varphi_i\>|\geq \frac{1}{\sqrt{N}}$.
\end{enumerate}
\end{theorem}

In order to prove Theorem \ref{t:M=N}, we need the following lemma.

\begin{lemma}\label{lm:exist}
Let $\Phi\in \R^{N\times N}$ be a non-unitary invertible matrix with unit norm columns. Then there exists a vector $d\in\R^N$ with $\|d\|_2 > 1$ and a vector $a\in\R^N$ with $|a_i|=1/\sqrt N$ for all $i=1,\ldots,N$,
such that $\Phi^Td = a$.
\end{lemma}
\begin{proof}
Let $\{b_i\}_{i=1} ^N$ be a sequence with each entry being a Bernoulli random variable, $\Psi = \diag(b_i)_{i=1}^N$, and $g=\frac{1}{\sqrt N}(1,...., 1)^T$. Suppose  $d_{\Psi}$ is the solution to
\begin{equation}\label{eq:linear}
\Phi^Td_{\Psi}=\Psi g.
\end{equation}
Let $\Phi^T=U\Sigma V^T$ be the singular value decomposition of $\Phi^T$, where $\Sigma=\text{diag}(\sigma_i)$. Observe that
\[
\sqrt N=\|\Phi\|_F=\|\Sigma\|_F=\sqrt{\sum\limits_{i=1}^{N}\sigma_i^2}.
\]
Hence, from \eqref{equ:ine} we obtain
$$
\sum_{i=1}^N\sigma_i^{-2} \ge\frac{N^2}{\sum_{i=1}^N\sigma_i^2} = N.
$$

On the other hand, from \eqref{eq:linear} we have
\[
V^Td_{\Psi}=\Sigma^{-1}U^T\Psi g.
\]
Next, we calculate the expectation $\Eb\|d_{\Psi}\|^2$. If it is greater than 1, then there must exist one instance of $d_{\Psi}$ with norm greater than 1, which makes the lemma hold. As $\Eb(b_ib_j) = \delta_{ij}$, we obtain
\begin{align*}
\Eb \|d_{\Psi}\|^2
&= \Eb \|V^Td_{\Psi}\|^2=\Eb \|\Sigma^{-1} U^T \Psi g\|^2\\
&= \frac 1 N\,\Eb \left(\sum\limits_{i}\sigma_i^{-2}\left(\sum\limits_{j}u_{ji}b_j\right)^2\right)\\
&= \frac 1 N\,\sum_i\sigma_i^{-2}\Eb\left(\sum_{j}u_{ji}^2 + \sum_j\sum_{k,k\neq j}u_{ji}u_{ki}b_jb_k\right)\\
&= \frac 1 N\,\sum_i\sigma_i^{-2}\sum_{j}u_{ji}^2 =\frac{1}{N}\sum\limits_{i}\sigma_i^{-2}\geq 1,
\end{align*}
while for the last inequality, equality holds only when all $\sigma_i$ are equal, i.e., $\Phi$ is unitary, which is ruled out by our assumption. Therefore the last inequality is strict.
\end{proof}

\begin{proof}[Proof of Theorem {\rm\ref{t:M=N}}]
The equivalence (i)$\Leftrightarrow$(ii) is easy to see and follows from, e.g., \cite[Corollary 2.8]{kopt13}. Moreover, (i)$\Rightarrow$(iii) is a direct consequence of Theorem \ref{pro:ns}(a). It remains to prove that (iii) implies (i). For this, we prove the contraposition. Suppose that $\Phi$ is not scalable. Then $\Phi$ is not unitary, and Lemma \ref{lm:exist} implies the existence of $d'\in\R^N$, $\|d'\|_2 > 1$, such that $|\<d',\vphi_i\>| = 1/\sqrt{N}$ for all $i=1,\ldots,N$. Hence, with $d = d'/\|d'\|_2$ we have $|\<d,\vphi_i\>| = (\|d'\|_2\sqrt{N})^{-1} < 1/\sqrt{N}$ for all $i=1,\ldots,N$. That is, (iii) does not hold, and the theorem is proved.
\end{proof}

\subsection{Estimation of the probability}\label{ss:prob}
With the help of Theorem \ref{pro:ns}, we now  estimate the probability for a frame to be  scalable when its vectors are drawn independently and uniformly from $\Sb^{N-1}$. First of all, it is easy to see the probability strictly increases as $M$ increases. Secondly, $\varphi_i\varphi_i^T\in\Sym_N$, where
$$
\Sym_N := \left\{A\in\R^{N\times N} : A = A^T\right\},
$$
which is a vector space of dimension $\frac{N(N+1)}{2}$. By \eqref{equ:scalar}, being scalable requires $I$ to be in the positive cone generated by $\{\varphi_i\varphi_i^T\}_{i=1}^M$. If $M<\frac{N(N+1)}{2}$, then this set  cannot be a basis of $\Sym_N$, so the chance for any symmetric matrix to be in the span of $\{\varphi_i\varphi_i^T\}_{i=1}^M$ is minimal,  which makes it even more difficult for $I$ to be in positive cone generated by this set. Therefore we expect the probability to be 0 when $M<\frac{N(N+1)}{2}$. Finally, as $M\rightarrow\infty$, we expect the probability of frames to be scalable to approach 1.

Let us first consider the case $N=2$ for which  the probability $P_{2,M}$ can be explicitly computed.

\begin{example}\label{exa:2}
If vectors $\varphi_1,\ldots,\vphi_M$ are drawn independently and uniformly from $\mathbb{S}^1$, then the probability of $\{\vphi_i\}_{i=1}^M$ to be a scalable frame in $\Fc_u(M,2)$ is given by
$$
P_{M,2}=1-\frac{M}{2^{M-1}},\quad M\geq2.
$$
\begin{proof}
First of all, define the angle of a vector $v$ as the angle between $v$ and positive $x$-axis, counterclockwise.
Among all the double cones that cover all the vectors in $\Phi_{\rm Sym}$, let $P_\Phi$ be the one with the smallest apex angle $\alpha$. It is known that $\Phi$ is scalable if and only if $\alpha\ge\pi/2$. Let $\varphi_\Phi$ be the ``right boundary'' of $P_\Phi$. To be rigorous,
Let $\varphi_\Phi$ be the vector with angle $\beta_0\in [0,\pi)$ such that for $\beta$ in some neighborhood of $\beta_0$ we have $(\cos\beta,\sin\beta)^T\in P_\Phi$ if $\beta > \beta_0$ and $(\cos\beta,\sin\beta)^T\notin P_\Phi$ if $\beta < \beta_0$. For fixed $i\in\{1,\ldots,M\}$ we then have
\begin{align*}
\Pr(\Phi&\text{ not scalable and } \varphi_\Phi=\pm\varphi_i )\\
&=\frac 1{2\pi}\int_0^{2\pi}\Pr\left( \Phi\text{ not scalable and }\varphi_\Phi = \pm\varphi_i\,|\,\angle\varphi_i = \beta \right)\,d\beta\\
&=\frac 1{2\pi}\int_0^{2\pi}\frac{1}{2^{M-1}}\,d\beta = \frac{1}{2^{M-1}}.
\end{align*}
Now, it follows that $\Pr(\Phi\text{ is not scalable}) = \sum_i \Pr(\Phi\text{ not scalable and } \varphi_\Phi=\pm\varphi_i )=M/2^{M-1}$.
\end{proof}
\end{example}

We can see in $\R^2$, as the number of frame vectors increases, the probability $P_{M,2}$ increases as well, starting from zero and eventually approaching 1. The critical point where the probability turns from zero to positive is $M=3=\frac{N(N+1)}{2}$, which meets our expectation. We will show that this is true for arbitrary dimension, and provide an estimate for the probability of frames being scalable. The following lemma completes the series of preparatory statements for the proof of our main theorem.

\begin{lemma}\label{lm_open}
If $\Phi=\{\varphi_i\}_{i=1}^M$ is a strictly scalable frame for $\R^N$ and $\{\varphi_i\varphi_i^T\}_{i=1}^M$ is a frame for $\Sym_N$, then there exists $\varepsilon > 0$ such that any frame $\Psi$ satisfying $\|\Psi-\Phi\|_F < \varepsilon$ is strictly scalable.
\end{lemma}
\begin{proof}
Let $A$ be the lower frame bound of $\{\vphi_i\vphi_i^T\}_{i=1}^M$, where $\Sym_N$ is endowed with the Frobenius norm. Moreover, by $F : \Diag_M\to\Sym_N$ denote the synthesis operator of $\{\varphi_i\varphi_i^T\}_{i=1}^M$, where $\Diag_M$ denotes the space of all diagonal matrices in $\Sym_M$. Then $FD = \Phi D\Phi^T$, $D\in\Diag_M$. Since $\Phi$ is strictly scalable, there exists a positive definite $D\in\Diag_M$ such that $FD = I$.

Let $\delta > 0$ be so small that whenever $\Delta\in\Diag_M$ with $\|\Delta\|_F\le \delta$, we have that $D + \Delta$ remains positive definite. Moreover, let $\veps > 0$ be so small that
$$
\tau := (\veps + 2\|\Phi\|_F)\veps\le\max\left\{\frac{\sqrt A}2,\frac{\delta A}{2(\sqrt A + 2\|F\|_{\rm op})\|D\|_F}\right\}.
$$
Now, let $\Psi = \{\psi_i\}_{i=1}^M\subset\R^N$ be such that $\|\Phi - \Psi\|_F < \veps$. By $G : \Diag_M\to\Sym_N$ denote the synthesis operator of $\{\psi_i\psi_i^T\}_{i=1}^M$. We can see that $\|F - G\|_{\rm op}\le\tau$, since for any diagonal matrix $C$,
\begin{align*}
\|FC-GC\|_F&=\|\Phi C\Phi^T-\Psi C\Psi^T\|_F\leq\|\Phi C(\Phi^T-\Psi^T)\|_F+\|(\Phi-\Psi)C\Psi^T\|_F\\
&\leq \epsilon(\|\Phi\|_F+\|\Psi||_F)\|C\|_F\leq \epsilon(\epsilon+2\|\Phi\|_F)\|C||_F.
\end{align*}
 Hence, for $X\in\Sym_N$ we have
$$
\|G^*X\|_F\ge\|F^*X\|_F - \|(F-G)^*X\|_F\ge(\sqrt{A} - \tau)\|X\|_F\ge(\sqrt A/2)\|X\|_F.
$$
In particular, this implies that $\{\psi_i\psi_i^T\}_{i=1}^M$ is a frame for $\Sym_N$, and $\<GG^*X,X\>_F = \|G^*X\|_F^2\ge(A/4)\|X\|_F^2$ yields $\|(GG^*)^{-1}\|_{\rm op}\le 4/A$. Now, we define
$$
\Delta := G^*(GG^*)^{-1}(F-G)D\in\Diag_M.
$$
Then $G(D+\Delta) = GD + (F-G)D = FD = I$. Moreover,
$$
\|\Delta\|_F\le\|G\|_{\rm op}\|(GG^*)^{-1}\|_{\rm op}\|F-G\|_{\rm op}\|D\|_F\le((\sqrt A/2) + \|F\|_{\rm op})(4/A)\tau\|D\|_F\le\delta,
$$
so that $D + \Delta$ is positive definite. Consequently, $\Psi$ is strictly scalable.
\end{proof}

\begin{remark}\label{rem:open}
We mention that Lemma \ref{lm_open} implies that the set $\{\Phi\in SC_+(M,N):\{\varphi_i\varphi_i^T\}_{i=1}^M \text{ is a frame}\}$ is open.
\end{remark}

The statement and proof of the main theorem use the notion of  spherical caps. We define $R^N_a(C)$ to be the spherical cap in $\Sb^{N}$ with angular radius $a$, centered at $C$, i.e.
$$
R^N_a(C) = \left\{x\in\Sb^{N} : \<x,C\>\ge\cos(a)\right\}.
$$
\begin{center}
\includegraphics[width=0.2\textwidth]{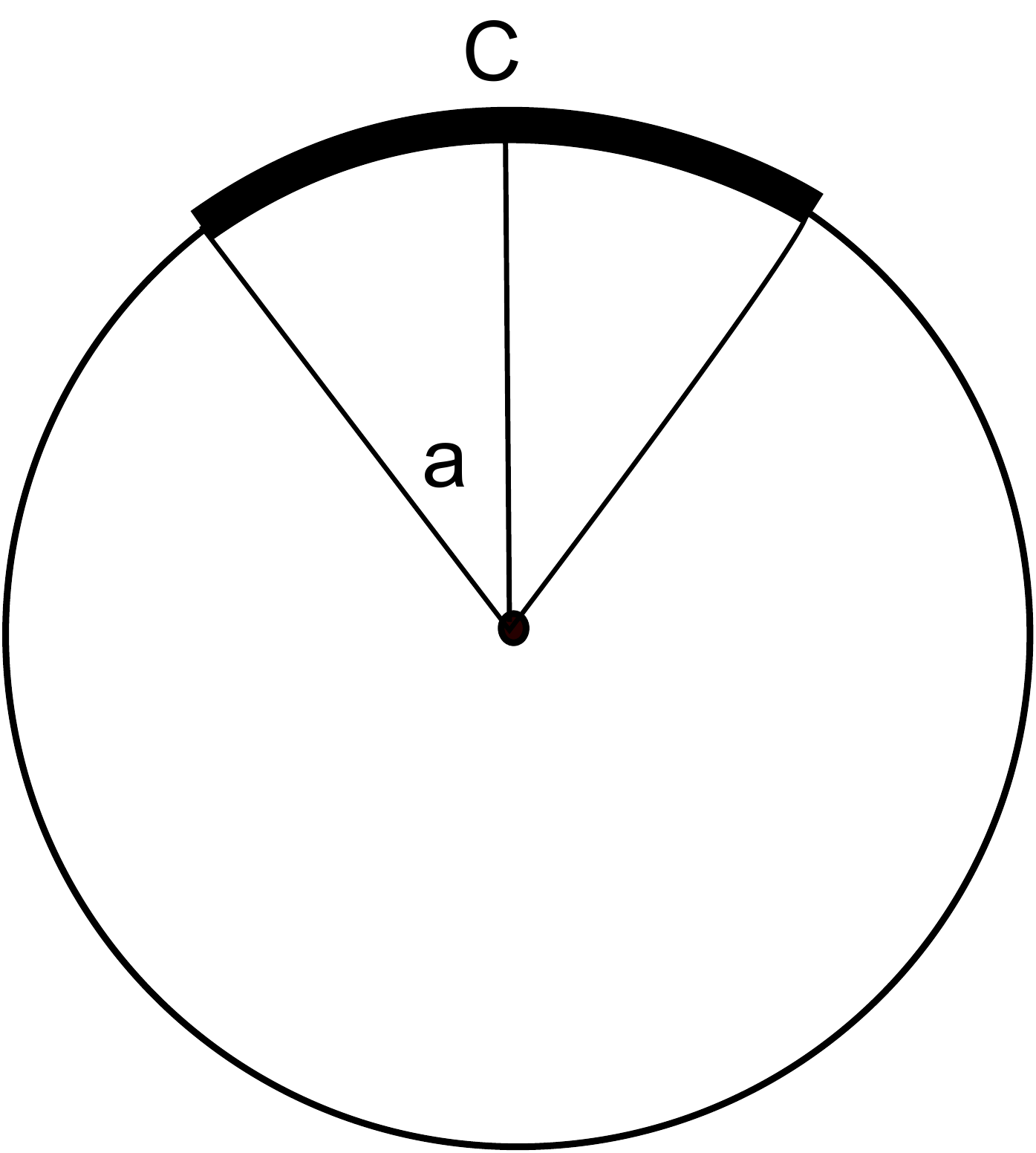}
\end{center}
By $A^N_a$ we denote the relative area  of $R^N_a(C)$ (ratio of area of $R^N_a(C)$ and area of $\Sb^{N}$).

\begin{theorem}\label{thm:pro}
Given $\Phi=\{\varphi_i\}_{i=1}^M\subset\R^N$, where each vector $\varphi_i$ is drawn independently and uniformly from $\mathbb{S}^{N-1}$, let $P_{M,N}$ denote the probability that $\Phi$ is scalable. Then the following holds:
\begin{enumerate}
\item[(i)]   When $M<\frac{N(N+1)}{2}$, $P_{M,N}=0$
\item[(ii)]  When $M\geq\frac{N(N+1)}{2}$, $P_{M,N}>0$ and
$$
C_N\left(1 - A_\alpha^{N-1}\right)^M\ge\,1 - P_{M,N}\,\ge\,\left(1-A^{N-1}_{a}\right)^{M-N},
$$
where
$$
\alpha = \frac{1}{2}\arccos\sqrt{\frac{N-1}{N}},\quad
a = \arccos\frac 1{\sqrt N},
$$
and where $C_N$ is the number of caps with angular radius $\alpha$ needed to cover $\Sb^{N-1}$. Consequently, $\lim_{M\rightarrow\infty}P_{M,N}=1$.
\end{enumerate}
\end{theorem}
\begin{proof}
By $\mu_u$ we denote the uniform measure on $\Sb^{N-1}$ and by $\mu_G$ the Gaussian measure on $\R^N$. Furthermore, on $(\Sb^{N-1})^M$ and $(\R^N)^M$ define the product measures
$$
\mu_u^k := \bigotimes_{j=1}^k\mu_u
\qquad\text{and}\qquad
\mu_G^k := \bigotimes_{j=1}^k\mu_G,
$$
respectively. For a set $B\subset(\Sb^{N-1})^k$, $k\in\N$, we define
$$
B' := \left\{(x_1,\ldots,x_k)\in(\R^N\setminus\{0\})^k : \left(\frac{x_1}{\|x_1\|_2},\ldots,\frac{x_k}{\|x_k\|_2}\right)\in B\right\}.
$$
Since $\mu_u(A) = \mu_G(A')$ for any $A\subset\Sb^{N-1}$, we have
\begin{equation}\label{e:Gauss_unif}
\mu_u^k(B) = \mu_G^k(B')\quad\text{for any }B\subset(\Sb^{N-1})^k.
\end{equation}

(i). Set $K = N(N+1)/2$. It suffices to show $P_{M,N} = 0$ only for $M = K-1$. For this, let
$$
B := \left\{(\vphi_1,\ldots,\vphi_M)\in(\Sb^{N-1})^M : \{\vphi_1\vphi_1^T,\ldots,\vphi_M\vphi_M^T,I\}\text{ is linearly dependent}\right\}.
$$
Then
$$
B' = \left\{(\vphi_1,\ldots,\vphi_M)\in(\R^N\setminus\{0\})^M : \{\vphi_1\vphi_1^T,\ldots,\vphi_M\vphi_M^T,I\}\text{ is linearly dependent}\right\}.
$$
This set, seen as a subset of $\R^{NM}$, is contained in the zero locus of a polynomial in the entries of the $\vphi_i$'s. Therefore, the Lebesgue measure of $B'$ is zero. But this shows that $\mu_G^M(B') = 0$ since $\mu_G^M$ is absolutely continuous with respect to the Lebesgue measure. Consequently, we obtain
$$
P_{M,N} = \mu_u^M(\{\Phi\in\fr_u(M,N) : \Phi\text{ scalable}\})\le\mu_u^M(B) = \mu_G^M(B') = 0.
$$

(ii). With Lemma \ref{lm_open}, we only need to prove the existence of a strictly scalable unit norm frame $\Phi$ such that $\{\varphi_i\varphi_i^T\}_{i=1}^M$ spans $\Sym_N$. For this, we note that by \cite[Theorem 2.1]{cc12}, there exists a frame $V=\{v_i\}_{i=1}^M$ such that $\{v_iv_i^T\}_{i=1}^M$ spans $\Sym_N$. Let $S$ be its frame operator, and $\varphi_i=S^{-1/2}v_i$. Therefore $\Phi=\{\varphi_i\}$ is a tight frame, thus strictly scalable. It is also easy to check that the  linear map $T: \Sym_N\to\Sym_N$, defined by $T(A) := S^{-1/2}AS^{-1/2}$, $A\in\Sym_N$, is invertible and maps $v_iv_i^T$ to $\varphi_i\varphi_i^T$. Therefore, $\{\varphi_i\varphi_i^T\}_{i=1}^M$ also spans $\Sym_N$. Finally, we normalize $\Phi$ to attain the desired frame.

For the estimate on $1 - P_{M,N}$, we first prove the right hand side inequality. For this, we put $\Psi := \{\vphi_i\}_{i=1}^N$ and $\Upsilon := \{\vphi_i\}_{i=N+1}^M$. If $\Psi$ is not unitary, by Theorem \ref{t:M=N} there exists $d_\Psi\in\Sb^{N-1}$ such that $|\<d_\Psi,\vphi_i\>| < 1/\sqrt N$ and hence $\vphi_i\notin R_a^{N-1}(d_\Psi)$ for $i=1,\ldots,N$. Therefore, if $\Psi$ is not unitary and $\vphi_{N+1},\ldots,\vphi_M\notin R_a^{N-1}(d_\Psi)$ then $\Phi$ is not scalable by Theorem \ref{pro:ns}(a). This yields
\begin{align*}
1 - P_{M,N}
&\ge\mu_u^M\left(\left\{\Phi : \Psi\notin\sca(N,N),\;\forall\vphi\in\Upsilon : \vphi\notin R_a^{N-1}(d_\Psi)\right\}\right)\\
&= \int\limits_{\sca(N,N)^c}\mu_u^{M-N}\left(\left\{\Upsilon\in(\Sb^{N-1})^{M-N} : \forall\vphi\in\Upsilon : \vphi\notin R_a^{N-1}(d_\Psi)\right\}\right)\,d\mu_u^N(\Psi)\\
&= \int\limits_{\sca(N,N)^c}\mu_u^{M-N}\left(\left(\left[R_a^{N-1}(d_\Psi)\right]^c\right)^{M-N}\right)\,d\mu_u^N(\Psi)\\
&= \left(1 - A_a^{N-1}\right)^{M-N}\int\limits_{\sca(N,N)^c}\,d\mu_u^N(\Psi)\\
&= \left(1 - A_a^{N-1}\right)^{M-N}\left(1-\mu_u^N(\sca(N,N))\right).
\end{align*}
But $\mu_u^N(\sca(N,N)) = 0$ by (i), and hence the inequality follows.

For the left hand side inequality, let  $\{R_j\}_{j=1}^C$ be a cover of $\Sb^{N-1}$ with spherical caps of angular radius $\alpha$. Define the event $E:=\{\forall j\in\{1,2,\cdots,C\}\exists i \text{ such that } \vphi_i\in R_j\}$. If event $E$ holds, whenever $d\in\Sb^{N-1}$, there exists $j$ such that $d\in R_j$. Thus, there also exists $i$ such that $d$ and $\varphi_i$ are in the same spherical cap, which means $\langle d,\varphi_i\rangle\geq\sqrt{\frac{N-1}{N}}$. Therefore, Theorem \ref{pro:ns}(b) yields that $\Phi$ is scalable. So, we have
\begin{align*}
P_{M,N}
&\ge \Pr(E) = 1 - \Pr\left(\exists j\,\forall i : \vphi_i\in R_j^c\right)\\
&= 1 - \Pr\left(\bigcup_j\{\forall i : \varphi_i\in\R_j^c\}\right) \\
&\geq 1-\sum_j\Pr\left(\{\forall i : \varphi_i\in\R_j^c\}\right)\\
&= 1 - \sum_j \left(1 - A_\alpha^{N-1}\right)^M = 1 - C\left(1 - A_\alpha^{N-1}\right)^M.
\end{align*}
This finishes the proof of the theorem.
\end{proof}

\begin{remark}
An upper bound on $C_N$ can be found in \cite[Theorem 1.2]{BCL10} as
\[
C_N\leq 3N+2+\sqrt{N}(N+1)\cos(a)(A^{N-1}_{a})^{-2}\left(\frac{1}{2A^{N-1}_{a}}\right)^{N}.
\]
\end{remark}

\section{Acknowledgments}
G. Kutyniok acknowledges support by the Einstein Foundation Berlin,
by the Einstein Center for Mathematics Berlin (ECMath), by Deutsche Forschungsgemeinschaft (DFG) Grant KU 1446/14, by the DFG Collaborative Research Center SFB/TRR 109 "Discretization in Geometry and Dynamics", and by the DFG Research Center {\sc Matheon} "Mathematics for key technologies" in Berlin. Also F. Philipp thanks the {\sc Matheon} for their support. K.~A.~Okoudjou was supported by the Alexander von Humboldt foundation. He would also like to express his gratitude to  the Institute of Mathematics at the Technische Universit\"at Berlin for its hospitality while part of this work was completed. R.~Wang was supported by CRD Grant DNOISE 334810-05 and by the industrial sponsors of the Seismic Laboratory for Imaging and Modelling: {BG Group}, BGP, {BP}, {Chevron}, ConocoPhilips, {Petrobras}, PGS, Total SA, and {WesternGeco}. Furthermore, the authors thank Anton Kolleck (TU Berlin) for valuable discussions.


\begin{thebibliography}{10}
\bibitem{kb97}
K.~Ball,
{\it An elementary introduction to modern convex geometry,}
Flavors of geometry, 1--58, Math.\ Sci.\ Res.\ Inst.\ Publ., {\bf  31}, Cambridge Univ. Press, Cambridge, 1997.


\bibitem{BF}
J.~J.~Benedetto and M.~Fickus,
{\it Finite Normalized Tight Frames,}
Adv. Comput. Math., {\bf 18} (2003), 357--385.

\bibitem{BCL10}
P.~B\"urgisser, F.~Cucker, and M.~Lotz,
{\it Coverage processes on spheres and condition numbers for linear programming,}
The Annals of Probability, {\bf 38.2} (2010): 570--604.

\bibitem{cc12}
J.~Cahill and X.~Chen,
{\it A note on scalable frames},
Proceedings of the 10th International Conference on Sampling Theory and Applications, pp.\ 93--96.

\bibitem{cfmmn}
J.~Cahill, M.~Fickus, D.~G.~Mixon, M.~J.~Poteet, and N.~Strawn,
{\it  Constructing finite frames of a given spectrum and set of lengths,}
Appl. Comput. Harmon. Anal., {\bf 35} (2013), no. 1, 52--73.




\bibitem{cfm11}
P.~G.~Casazza, M.~Fickus, and D.~G.~Mixon,
{\it Auto-tuning unit norm frames,}
Appl.\ Comput.\ Harmon.\ Anal., {\bf  32} (2012), no. 1, 1-–15.


\bibitem{ck12}
P.~G.~Casazza and G.~Kutyniok,
Finite Frame Theory, Eds., Birkh\"auser, Boston (2012).

\bibitem{cl10}
P. G. Casazza and M. Leon.
{\it Existence and construction of finite frames with a given frame operator.}
Int. J. Pure Appl. Math, {\bf 63} (2010), 149--158.

\bibitem{co03}
O.~Christensen,
An introduction to frames and Riesz bases,
Applied and Numerical Harmonic Analysis. Birkh\"auser Boston, Inc., Boston, MA, 2003.

\bibitem{cklmns12}
M.~S.~Copenhaver, Y.~H.~Kim, C.~Logan, K.~Mayfield, S.~K.~Narayan, and J.~Sheperd,
{\it Diagram vectors and tight frame scaling in finite dimensions,}
Operators and Matrices,  {\bf 8}, no.1 (2014), 73 -- 88.

\bibitem{fjko}
M.~Fickus, B.~D.~Johnson, K.~Kornelson, and K.~A.~Okoudjou,
{\it Convolutional frames and the frame potential,}
Appl. Comput. Harmon. Anal., { \bf 19} (2005), 77--91.



\bibitem{og10}
O.~G\"uler,
Foundations of Optimization,
Graduate Texts in Mathematics, {\bf 258} Springer, New York, 2010.


\bibitem{fj48}
F.~John,
{\it Extremum problems with inequalities as subsidiary conditions,}
Studies and Essays Presented to R.\ Courant on his $60^{th}$ Birthday, January 8, 1948, 187--204.
Interscience Publishers, Inc., New York, N. Y., 1948.

\bibitem{kha}
L.~G.~Khachiyan,
{\it Rounding of polytopes in the real number model of computation,}
Math. Oper. Res., {\bf  21}, 1996, 307--320.




 \bibitem{kumar}
P.~Kumar, E.~A.~Yildirim,
{\it Minimum volume enclosing ellipsoids and core sets},
J. Optim. Theory Appl., {\bf 126} (2005), 1--21.



\bibitem{kop13c}
G.~Kutyniok, K.~A.~Okoudjou, and F.~Philipp,
{\it Scalable frames and convex geometry,}
Spectra of Wavelets, Tilings, and Frames (Boulder, CO, 2012),
Contemp. Math. 345, Amer. Math. Soc., Providence, RI (2013), to appear.

\bibitem{kopt13}
G.~Kutyniok, K.~A.~Okoudjou, F.~Philipp, and E.~K.~Tuley,
{\it Scalable frames},
Linear Algebra and its Applications {\bf 438} (2013), 2225--2238.


\bibitem{NDEG13}
S. Nam, M. E. Davies, M. Elad, and R. Gribonval,
{\it The Cosparse Analysis Model and Algorithms},
Appl. Comput. Harmon. Anal., {\bf 34} (2013), 30--56.


\bibitem{np06}
C~P.~Niculescu and L.-E.~Persson,
Convex Functions and Their Applications -- A Contemporary Approach,
Canadian Mathematical Society, Springer, New York, 2006.


\bibitem{RZE10}
R. Rubinstein, M. Zibulevsky, and M. Elad,
{\it Double Sparsity: Learning Sparse Dictionaries for Sparse Signal Approximation},
IEEE Trans. Signal Process., {\bf 58} (2010), 1553--1564.

\bibitem{str12}
N.~Strawn,
{\it Optimization over finite frame varieties and structured dictionary design,}
Appl. Comput. Harmon. Anal., {\bf 32} (2012), 413--434.

\bibitem{ntj89}
N.~Tomczak-Jaegermann,
Banach-Mazur Distances and Finite-Dimensional Operator Ideals,
Pitman Monographs and Surveys in Pure and Applied Mathematics, {38} Longman Scientific $\&$ Technical, Harlow;
copublished in the United States with John Wiley $\&$ Sons, Inc., New York, 1989.

\end{thebibliography}
\end{document}